\documentclass{amsart}

\usepackage[margin=2.5cm]{geometry}
\usepackage{amssymb,amsmath,upgreek,mathrsfs,bm,cases,latexsym,
	enumerate,graphicx,algorithm,array,booktabs,mathtools,placeins,xcolor,float}
	\usepackage{algorithm}
	\usepackage{graphicx}
	\usepackage{booktabs}
	\usepackage{tabularx}
\usepackage{algpseudocode}
\floatname{algorithm}{Algorithm}

\DeclareMathOperator*{\argmin}{arg\,min}

\makeatletter
\providecommand*{\toclevel@algorithm}{0}
\makeatother

\usepackage{hyperref}
\usepackage{xspace}
\hypersetup{
	colorlinks=true,
	linkcolor={blue},
	citecolor={blue},
	urlcolor={blue}
}

\newcommand{\probref}[2]{\hyperref[#1]{\textup{#2}}\xspace}

\newtheorem{theorem}{ \bf Theorem}[section]
\newtheorem{proposition}{ \bf Proposition}[section]
\newtheorem{lemma}{ \bf Lemma}[section]
\newtheorem{remark}{ \bf Remark}[section]
\newtheorem{definition}{ \bf Definition}[section]

\newtheorem{corollary}{ \bf Corollary}[section]
\numberwithin{equation}{section}
\allowdisplaybreaks[4]
\begin{document}
	\title[Accelerated Gradient Flow ]{Accelerated Gradient Flow with Endogenous Gradient-Memory Anchor: Selection and Restoring Damping}

	\author[C. Izuchukwu]{Chinedu Izuchukwu}
	\address[C. Izuchukwu]{School of Mathematics, University of the Witwatersrand, Private Bag 3, Johannesburg, 2050, South Africa}
	\email{chinedu.izuchukwu@wits.ac.za}

		\author[E. Obini]{Ernest Obini}
	\address[E. Obini]{School of Mathematics, University of the Witwatersrand, Private Bag 3, Johannesburg, 2050, South Africa}
	\email{3181824@students.wits.ac.za}

		\author[J. Yao]{Jen-Chih Yao}
	\address[J. Yao]{Center for General Education, China Medical University, Taichung 40402, Taiwan and Academy of Romanian Scientists, Bucharest, Romania}
	\email{yaojc@mail.cmu.edu.tw}

		\author[S. Zeng]{Shengda Zeng}
	\address[S. Zeng]{National Center for Applied Mathematics in Chongqing, Chongqing Normal University, Chongqing 401331, China}
	\email[Corresponding author]{zengshengda@163.com}

	\begin{abstract}
	We introduce an accelerated gradient flow with an endogenous gradient-memory anchor and a nonlinear restoring--damping feedback. The anchor evolves from a weighted history of the gradients, while the nonlinear feedback is modulated by the squared coordinatewise displacement from the anchor. The resulting dynamics uses only first-order information from the objective function and does not involve explicit Hessian information. Under suitable assumptions, we establish global existence and uniqueness of strong solutions. A Lyapunov analysis yields the accelerated objective-value estimate $F(x(t))-F^\star=\mathcal{O}(t^{-2})$, with the nonlinear restoring--damping term contributing additional dissipation. Under similar assumptions, the primal trajectory converges strongly to a minimizer. When, in addition, the solution set $S$ is affine, the limit is identified explicitly as the Euclidean projection $P_S(z_0)$ of the initial anchor $z_0$ onto $S$. In particular, for rank-deficient least-squares problems, the dynamics selects the least-squares solution closest to $z_0$. We also establish a finite weighted dissipation estimate for the phase variable. Moreover, in coordinates where the nonlinear feedback is active, if the trajectory remains separated from the anchor over an interval, then the corresponding homogeneous phase dynamics acquires an additional polynomial decay factor. Numerical experiments illustrate the minimizer-selection, accelerated objective-value decay, and restoring--damping mechanisms.
	\end{abstract}
	
	\maketitle
	
	
	
	\section{Introduction}
	\label{sec:introduction}
	
	\subsection{Background and Motivation}
	\label{subsec:background-motivation}
	Continuous-time dynamical systems provide a useful framework for understanding and designing accelerated optimization methods \cite{SuBoydCandes16,WWJ}. A fundamental example is the differential equation associated with Nesterov's accelerated gradient method \cite{SuBoydCandes16},
	\begin{equation}\label{NODE}
		\ddot{x}(t)
		+
		\frac{\alpha}{t}\dot{x}(t)
		+
		\nabla F(x(t))
		=
		0,
		\qquad t>0,
	\end{equation}
	which, for a convex and continuously differentiable objective function $F$ and suitable values of $\alpha$, enjoys the accelerated objective-value estimate
	$$
	F(x(t))-F^\star=\mathcal{O}(t^{-2}).
	$$
	This continuous-time viewpoint has led to several extensions of accelerated dynamics, with particular attention to their convergence, stability, and asymptotic behavior \cite{Alvarez00, pol, WWJ}.
	
	The first issue relevant to the present work is \emph{stabilization}. Inertial dynamics may exhibit oscillatory transients, particularly when the damping is weak or the problem is poorly conditioned. From an optimization point of view, such oscillations are undesirable because the trajectory may repeatedly overshoot the minimizer, the objective values may become nonmonotone, and the system may spend a relatively long time oscillating around the solution set before settling. Similar effects are well known in accelerated first-order methods, where excessive momentum may lead to oscillatory or nearly periodic behavior; this is one of the motivations behind restart strategies and additional damping mechanisms \cite{ODC,AC,APR,ABC}.
	
	 Our focus is on damping mechanisms, notably Hessian-driven damping \cite{APR} and closed-loop damping \cite{ABC}. The Hessian-driven damping introduces curvature-dependent dissipation through terms involving
	$
	\nabla^2F(x(t))\dot{x}(t)
	$
	\cite{APR}, while closed-loop damping uses feedback from the current state of the system; in the main model of Attouch--Bo\c{t}--Csetnek \cite{ABC}, this feedback acts through the velocity. This naturally raises the question of whether one can obtain an effective displacement-dependent stabilization mechanism while still using only first-order information from the objective function.	

	The second issue is \emph{minimizer selection}. When $\arg\min F$ contains more than one point, convergence of the objective values does not determine which minimizer is eventually selected by the dynamics. 
	Vanishing Tikhonov regularization provides one way of addressing this question, including accelerated dynamics with strong convergence to the minimum-norm solution \cite{ACR,AlecsaLaszlo21,AttouchLaszlo24,AttouchChbaniRiahi26}. Fixed-anchor and Halpern-type dynamics instead use a prescribed reference point \cite{SPR}. These mechanisms select distinguished minimizers through externally prescribed regularization or anchoring.
	 
	In the present paper, we consider a different mechanism. We introduce an auxiliary state $z(t)$ whose evolution is determined by the gradient history of the trajectory. More precisely,
	$$
	z(t)
	=
	z_0
	-
	\frac{1-a}{\lambda}
	\int_{t_0}^{t}s\,\nabla F(x(s))\,ds,
	\qquad
	\lambda=\alpha-1.
	$$
	Thus, only the initial value $z_0$ is prescribed externally. After initialization, the state $z(t)$ evolves together with the trajectory and records a weighted history of the gradients. For this reason, we refer to $z$ as an \emph{endogenous gradient-memory anchor}. The anchor also provides a natural reference state for introducing additional stabilization. We therefore construct a nonlinear restoring--damping feedback whose gain depends on the coordinatewise displacement $x(t)-z(t)$. In each coordinate, this gain increases with the squared displacement from the anchor and becomes weaker as the displacement decreases. Importantly, this mechanism uses only first-order information from $F$ and does not involve explicit Hessian information.

	The proposed dynamics is therefore designed to combine three features: the accelerated $\mathcal{O}(t^{-2})$ objective-value behavior associated with Nesterov-type dynamics \cite{SuBoydCandes16}, an endogenous mechanism that can influence the asymptotic selection of minimizers, and a displacement-sensitive restoring--damping effect. Detailed comparisons with Nesterov's ODE, Hessian-driven damping, Tikhonov regularization, fixed-anchor schemes, and closed-loop damping are given in Section~\ref{sec:relation-existing-models}.
	
	\subsection{Main Contributions}
	\label{subsec:main-contributions}
	
	The main contributions of the paper are summarized as follows.
	\begin{enumerate}
		\item[\rm (i)]
		We introduce an accelerated gradient flow in which the accelerated phase is coupled to a moving anchor generated internally from weighted past gradients. The resulting feedback produces both damping and restoring effects, with a gain depending on the squared coordinatewise displacement from the anchor. In contrast with Hessian-driven damping, the proposed mechanism uses only first-order information from $F$ and does not involve explicit Hessian information.
		
		\item[\rm (ii)]
		We establish global existence and uniqueness of strong solutions, and construct a Lyapunov functional that gives
		$
		F(x(t))-F^\star=\mathcal{O}(t^{-2}),
		$
		while the nonlinear restoring--damping term appears with a favorable dissipative sign in the Lyapunov estimate. Hence the additional feedback does not destroy the accelerated objective-value rate.
		
		\item[\rm (iii)]
		
		When $\gamma>0$, we prove that the primal trajectory converges strongly to a minimizer of $F$. Moreover, if
		$
		S:=\arg\min F
		$
		is a nonempty affine set, then the limiting point is identified explicitly by $
		x(t)\longrightarrow P_S(z_0)
		~~\text{as }t\to\infty.$
		Thus, in the affine-solution setting, the initial anchor $z_0$ determines the minimizer selected by the dynamics. In particular, for rank-deficient least-squares problems, the trajectory converges to the least-squares solution closest to $z_0$ in Euclidean distance (Corollary \ref{cor:rank-deficient-least-squares-selection}).
		
		\item[\rm (iv)]
		We obtain finite weighted coordinatewise phase dissipation. In particular, persistent large phase amplitudes cannot be maintained, in a precise weighted sense, in coordinates that remain significantly separated from the anchor. Moreover,  we obtain an additional polynomial attenuation of the homogeneous phase response on intervals where the corresponding anchor displacement remains bounded away from zero. These estimates provide a quantitative description of the stabilizing effect of the nonlinear restoring--damping mechanism.
	\end{enumerate}
	\noindent
	The numerical experiments complement the theoretical analysis by illustrating the three main mechanisms developed in the paper: anchor-induced minimizer selection, preservation of the accelerated objective-value decay, and coordinatewise restoring--damping.

\subsection{Organization}
\label{subsec:organization}
The remainder of the paper is organized as follows. Section \ref{MC} develops the proposed dynamics. Section \ref{sec:relation-existing-models} discusses the relations of our proposed dynamics with existing optimization dynamics. Section \ref{EU} establishes existence and uniqueness of solutions. Section \ref{sec:convergence-selection-damping} presents the main convergence results, including Lyapunov analysis, accelerated objective-value convergence, trajectory convergence, minimizer selection, and quantitative restoring--damping effects. Section \ref{sec:numerical-experiments} presents numerical experiments, and Section \ref{Con} concludes the paper.

\section{Model Construction}\label{MC}
\noindent The starting point of our construction is the accelerated phase appearing in the classical Nesterov ODE Lyapunov analysis \cite{SuBoydCandes16}: 
$$
t\dot x(t)+(\alpha-1)\bigl(x(t)-x^\star\bigr).
$$
In our construction, we \emph{internalize} this expression by replacing the reference minimizer $x^\star$ appearing in the Lyapunov phase with an auxiliary state $z(t)$. Thus, setting
$
\lambda:=\alpha-1>0,
$
we introduce the phase variable
\begin{equation}\label{PR}
	y(t):=t\dot x(t)+\lambda\bigl(x(t)-z(t)\bigr).
\end{equation}
We call this identity the \emph{internalized accelerated phase relation}. Its role is to combine the scaled velocity $t\dot x(t)$ with the displacement from the anchor $\lambda(x(t)-z(t))$.
 We next specify how the phase acts on the primal motion. Let $a\in[1/2,1]$ and $\gamma\geq0$. We consider
\begin{equation}\label{MD}
	\ddot x(t)
	=
	-\frac{\lambda+1}{t}\dot x(t)
	-\frac{\gamma}{t}y(t)
	-a\nabla F(x(t)).
\end{equation}
Here $a$ controls the direct gradient forcing, while $\gamma$ determines the strength of the phase feedback. It remains to specify the evolution of the anchor $z(t)$. Differentiating the phase relation \eqref{PR}, we obtain
$$
\dot y(t)
=
\lambda\bigl(\dot x(t)-\dot z(t)\bigr)
+\dot x(t)
+t\ddot x(t).
$$
Using \eqref{MD}, this becomes
$$
\dot y(t)
=
-\lambda\dot z(t)
-\gamma y(t)
-a\,t\nabla F(x(t)).
$$
We couple the anchor to the gradient signal by choosing
\begin{equation}\label{zdot}
	\dot z(t)
	=
	-\frac{t}{\lambda}(1-a)\nabla F(x(t)).
\end{equation}
With this choice, the phase satisfies the closed identity
$$
\dot y(t)+\gamma y(t)
=
(1-2a)t\nabla F(x(t)).
$$
The parameters $a$ and $1-a$ determine, respectively, the direct gradient forcing in the primal equation and the gradient-memory contribution to the anchor evolution. This choice also produces the phase structure used in the Lyapunov analysis developed later.

 Integrating \eqref{zdot} from $t_0$ to $t$ gives
\begin{equation}\label{z}
	z(t)
	=
	z(t_0)
	-\frac{1-a}{\lambda}
	\int_{t_0}^{t}s\,\nabla F(x(s))\,ds.
\end{equation}
Thus, for $a<1$, the anchor is generated dynamically from a weighted history of the gradients, and we refer to it as an \emph{endogenous gradient-memory anchor}. At the endpoint $a=1$, the integral term vanishes and $z(t)\equiv z(t_0)$, giving the fixed-anchor case considered separately later. Combining \eqref{PR}, \eqref{MD}, and \eqref{zdot}, we obtain the following coupled time-dependent system:
\begin{equation}\label{eq:SYS}
	\left\{
	\begin{aligned}
		y(t) &= \lambda\bigl(x(t)-z(t)\bigr)+t\dot x(t),\\
		\ddot x(t) &= -\frac{\lambda+1}{t}\dot x(t)-\frac{\gamma}{t}y(t)-a\nabla F(x(t)),\\
		\dot z(t) &= -\frac{t}{\lambda}(1-a)\nabla F(x(t)),
	\end{aligned}
	\right.
	\qquad
	t\geq t_0>0.
\end{equation}
Substituting the phase relation into the acceleration equation shows that system \eqref{eq:SYS} is equivalently written as
\begin{equation}\label{eq:SYS2}
	\ddot x(t)
	+
	\left(\frac{\alpha}{t}+\gamma\right)\dot x(t)
	+
	\frac{\gamma\lambda}{t}\bigl(x(t)-z(t)\bigr)
	+
	a\nabla F(x(t))
	=
	0,
\end{equation}
together with the anchor equation \eqref{zdot}.
The phase feedback therefore produces two additional terms in the primal equation: the viscous damping term $\gamma\dot x(t)$ and the linear restoring term
$
\frac{\gamma\lambda}{t}\bigl(x(t)-z(t)\bigr).
$
The latter has the form of a vanishing linear restoring force toward the current anchor. These effects are isotropic with respect to the anchor displacement, since their coefficients do not depend on the individual coordinatewise displacements from the anchor.

\hfill

\noindent To introduce a displacement-sensitive correction, let
$
M=\operatorname{Diag}(m_1,\ldots,m_d)
$
with $m_i\geq0$, and for
$
u=(u_1,\ldots,u_d)\in\mathbb{R}^d,
$
define
\begin{equation}
	D_M(u)
	:=
	M\operatorname{Diag}(u_1^2,\ldots,u_d^2)M
	=
	\operatorname{Diag}\bigl(m_1^2u_1^2,\ldots,m_d^2u_d^2\bigr).
	\label{eq:DM-definition}
\end{equation}
Then $D_M(u)$ is positive semidefinite, and its $i$-th diagonal entry $m_i^2u_i^2$ represents a weighted squared coordinatewise displacement. Taking $u=x(t)-z(t)$ therefore makes the feedback gain depend on the squared coordinatewise displacement from the anchor.

 We consequently add to the acceleration equation in \eqref{eq:SYS} the nonlinear phase-feedback term
$
-\frac{\eta}{t^2}
D_M\bigl(x(t)-z(t)\bigr)y(t),
~
\eta\geq0.
$
This gives the modified phase-feedback system
\begin{equation}
	\begin{cases}
		y(t)
		=
		\lambda\bigl(x(t)-z(t)\bigr)+t\dot{x}(t),
		\\[1.5mm]
		\ddot{x}(t)
		=
		-\dfrac{\lambda+1}{t}\dot{x}(t)
		-\dfrac{\gamma}{t}y(t)
		-\dfrac{\eta}{t^2}
		D_M\bigl(x(t)-z(t)\bigr)y(t)
		-a\nabla F(x(t)),
		\\[3mm]
		\dot{z}(t)
		=
		-\dfrac{t}{\lambda}(1-a)\nabla F(x(t)).
	\end{cases}
	\label{eq:system7}
\end{equation}
\noindent Using \eqref{PR} and the anchor equation, the phase satisfies
$$
\dot y(t)
+
\left[
\gamma I
+
\frac{\eta}{t}
D_M\bigl(x(t)-z(t)\bigr)
\right]y(t)
=
(1-2a)t\nabla F(x(t)).
$$
Hence the effective damping coefficient in the $i$-th phase coordinate is
$$
\gamma
+
\frac{\eta m_i^2|x_i(t)-z_i(t)|^2}{t}.
$$
For $\eta>0$ and $m_i\neq0$, the nonlinear part of this coefficient increases quadratically with the coordinatewise displacement for fixed $t$. If $m_i=0$, the nonlinear feedback is inactive in the $i$-th coordinate.

Expanding the phase variable in the acceleration equation, system \eqref{eq:system7} is equivalently written:
\begin{equation}
	\begin{aligned}
		\ddot{x}(t)
		&+
		\left(\frac{\alpha}{t}+\gamma\right)\dot{x}(t)
		+
		\frac{\gamma\lambda}{t}\bigl(x(t)-z(t)\bigr)
		+
		\frac{\eta}{t}D_M\bigl(x(t)-z(t)\bigr)\dot{x}(t)
		\\
		&\quad
		+
		\frac{\eta\lambda}{t^2}
		D_M\bigl(x(t)-z(t)\bigr)\bigl(x(t)-z(t)\bigr)
		+
		a\nabla F(x(t))
		=
		0,
	\end{aligned}
	\label{eq:system8}
\end{equation}
together with
\begin{equation}
	\dot{z}(t)
	=
	-\frac{t}{\lambda}(1-a)\nabla F(x(t)).
	\label{eq:z-dynamics}
\end{equation}
\noindent The term
$$
\frac{\eta}{t}D_M\bigl(x(t)-z(t)\bigr)\dot{x}(t)
$$
is a nonlinear state-dependent damping term. Since $t>0$, $\eta\geq0$, and $D_M(x-z)\succeq0$, we have
\begin{equation}
	\begin{aligned}
		\left\langle
		\frac{\eta}{t}
		D_M\bigl(x(t)-z(t)\bigr)\dot{x}(t),
		\dot{x}(t)
		\right\rangle
		&=
		\frac{\eta}{t}
		\sum_{i=1}^d
		m_i^2
		\bigl(x_i(t)-z_i(t)\bigr)^2
		\dot{x}_i(t)^2
		\geq0.
	\end{aligned}
	\label{eq:DM-dissipation}
\end{equation}
Thus its coordinatewise dissipative weight is
$$
\frac{\eta m_i^2}{t}|x_i(t)-z_i(t)|^2,
$$
which, for fixed $t$, increases quadratically with the coordinatewise displacement.
On the other hand, the term
$$
\frac{\eta\lambda}{t^2}
D_M\bigl(x(t)-z(t)\bigr)\bigl(x(t)-z(t)\bigr)
$$
is the corresponding nonlinear restoring term. Componentwise, it is given by
\begin{equation}
	\begin{aligned}
		&\frac{\eta\lambda}{t^2}
		D_M\bigl(x(t)-z(t)\bigr)
		\bigl(x(t)-z(t)\bigr)
		=
		\frac{\eta\lambda}{t^2}
		\bigl(
		m_1^2(x_1(t)-z_1(t))^3,\ldots,
		m_d^2(x_d(t)-z_d(t))^3
		\bigr).
	\end{aligned}
	\label{eq:nonlinear-restoring-componentwise}
\end{equation}
Since this term appears on the left-hand side of \eqref{eq:system8}, its contribution to the acceleration has the opposite sign. Hence, in each active coordinate, it acts toward the current anchor, with magnitude proportional to $|x_i(t)-z_i(t)|^3$ for fixed $t$.
Thus, the nonlinear phase feedback adds a displacement-sensitive damping and restoring mechanism to the linear phase-feedback system \eqref{eq:SYS2}. When $\eta=0$, system \eqref{eq:system7} reduces to \eqref{eq:SYS}. Moreover, the objective function enters the construction only through $\nabla F$, so the proposed mechanism does not involve explicit Hessian information.

\section{Relations with Existing Optimization Dynamics}
\label{sec:relation-existing-models}

\noindent A natural starting point for comparison is the Nesterov ODE of Su--Boyd--Cand\`es \cite{SuBoydCandes16}, given in \eqref{NODE}. Our system \eqref{eq:system7} recovers \eqref{NODE} when
$
a=1,\ \gamma=0,\ \eta=0.
$
Indeed, in this case the anchor is fixed and decouples from the primal equation. The classical Nesterov ODE contains no anchor state and therefore no corresponding mechanism through which an initial reference point can influence the asymptotic selection of a minimizer when $\argmin F$ is not a singleton.
 The proposed system \eqref{eq:system7} is also related to the Hessian-driven damping dynamics of Attouch--Peypouquet--Redont \cite{APR},
\begin{equation}\label{Hessian}
	\ddot{x}(t)
	+
	\frac{\alpha}{t}\dot{x}(t)
	+
	\beta\nabla^2F\bigl(x(t)\bigr)\dot{x}(t)
	+
	\nabla F\bigl(x(t)\bigr)
	=
	0.
\end{equation}
The additional term
$
\beta\nabla^2F(x(t))\dot{x}(t)
$
introduces curvature-dependent damping. By contrast, the nonlinear feedback in \eqref{eq:system7} does not involve explicit Hessian information. Its nonlinear gain is determined by the displacement from the endogenous anchor. More precisely, the $i$-th phase coordinate has the effective damping coefficient
$$
\gamma
+
\frac{\eta m_i^2|x_i(t)-z_i(t)|^2}{t}.
$$
Thus, for $\eta>0$ and $m_i\neq0$, the nonlinear part of the phase damping increases with the squared coordinatewise displacement from the anchor for fixed $t$.

 A further comparison is with Tikhonov-regularized inertial dynamics of the form
\begin{equation}\label{TRID}
	\ddot{x}(t)
	+
	\frac{\alpha}{t}\dot{x}(t)
	+
	\nabla F(x(t))
	+
	\varepsilon(t)x(t)
	=
	0,
	\qquad
	\varepsilon(t)\geq0,
	\qquad
	\varepsilon(t)\to0,
\end{equation}
as studied, for example,  in \cite{ACR,AttouchLaszlo24} (see also \cite{BCL,Las} for related Tikhonov-regularized systems). The term $\varepsilon(t)x(t)$ is the gradient of the time-dependent quadratic regularizer
$$
\frac{\varepsilon(t)}{2}\|x\|^2,
$$
and therefore introduces an externally prescribed attraction toward the origin. Under suitable assumptions, vanishing Tikhonov regularization yields strong minimum-norm selection, and a shifted quadratic regularizer can similarly target the solution closest to a prescribed state \cite{AttouchLaszlo24}. The distinction here is that no external regularization schedule is imposed: $z(t)$ is generated endogenously from weighted gradient history. Under the assumptions of Theorem~\ref{thm:anchor-induced-projection-selection}, when
$
S:=\argmin F
$
is a nonempty affine subspace,
$$
x(t)\longrightarrow P_S(z_0).
$$
 The proposed dynamics is also related to anchor-acceleration and Halpern-type continuous-time systems. Suh--Park--Ryu \cite{SPR} study the differential inclusion
\begin{equation}\label{FAH}
	\dot{x}(t)
	\in
	-A(x(t))
	-
	\beta(t)\bigl(x(t)-x_0\bigr),
\end{equation}
where the anchor $x_0$ is prescribed and remains fixed throughout the evolution. In contrast, only the initial value $z_0$ is prescribed in our system. Thereafter,
$$
\dot z(t)
=
-\frac{1-a}{\lambda}\,t\nabla F(x(t)),
$$
and hence
$$
z(t)
=
z_0
-
\frac{1-a}{\lambda}
\int_{t_0}^{t}
s\nabla F(x(s))\,ds.
$$
Thus, $z(t)$ is a trajectory-dependent state generated from the past gradient signal rather than a fixed external reference point. Moving-anchor extragradient schemes have also been studied for structured minimax problems \cite{AlcalaChowSunkula23}; their setting and update mechanism are different from the present continuous gradient-memory anchor.

The proposed system is also related to inertial dynamics with closed-loop damping. Attouch--Bo\c{t}--Csetnek \cite{ABC} consider, as a principal model,
\begin{equation}\label{CL}
	0
	\in
	\ddot{x}(t)
	+
	\partial\phi\bigl(\dot{x}(t)\bigr)
	+
	\nabla F\bigl(x(t)\bigr),
\end{equation}
where the nonlinear damping law $\partial\phi(\dot{x}(t))$ acts through the velocity. By contrast, the feedback in \eqref{eq:system7} is nonautonomous and acts on the accelerated phase
$$
y(t)
=
t\dot{x}(t)
+
\lambda\bigl(x(t)-z(t)\bigr)
$$
through
$$
-\frac{\gamma}{t}y(t)
-
\frac{\eta}{t^2}
D_M\bigl(x(t)-z(t)\bigr)y(t).
$$
The nonlinear gain is therefore modulated by the squared coordinatewise displacement from the endogenous anchor.
 These comparisons highlight the main structural distinction of the proposed dynamics. It combines Nesterov-type inertial motion with an endogenous gradient-memory anchor and a displacement-sensitive restoring--damping feedback. The results established later show that, for
$
\frac12<a<1,~\alpha\geq3,~ \gamma,\eta\geq0,
$
under the stated assumptions on $F$, this structure preserves the accelerated objective-value estimate
$$
F(x(t))-F^\star=\mathcal O(t^{-2}).
$$
When $\gamma>0$, the primal trajectory converges strongly to a minimizer. When, in addition, the solution set $S$ is affine, the limit is identified as $P_S(z_0)$. When $\eta>0$, the nonlinear term also yields a weighted phase-dissipation estimate and, for coordinates with $m_i\neq0$, an additional polynomial attenuation of the homogeneous phase response on intervals where the corresponding anchor displacement remains bounded away from zero; see Subsection~\ref{subsec:quantitative-restoring-damping}.

\begin{table}[htbp]
	\centering
	\small
	\caption{Comparison with closely related accelerated and inertial dynamics.}
	\label{tab:compact-model-comparison}
	\begin{tabularx}{\textwidth}{
			p{0.21\textwidth}
			p{0.25\textwidth}
			X
		}
		\toprule
		\textbf{Model class}
		&
		\textbf{Main mechanism}
		&
		\textbf{Structural distinction from the present work}
		\\
		\midrule
		
		Nesterov's ODE \eqref{NODE}
		&
		Vanishing viscous damping and accelerated objective-value decay.
		&
		It contains no anchor or gradient-memory state and therefore no corresponding
		anchor-based selection mechanism.
		\\\\[0.8ex]
		
		Hessian-driven damping \eqref{Hessian}
		&
		Curvature-dependent damping through
		$\nabla^{2}F(x(t))\dot{x}(t)$.
		&
		The present nonlinear feedback does not involve explicit Hessian information
		and has a gain modulated by the squared coordinatewise displacement from the
		endogenous anchor.
		\\\\[0.8ex]
		
		Tikhonov-regularized inertial dynamics \eqref{TRID}
		&
		Vanishing regularization and, under suitable assumptions, minimum-norm selection.
		&
		The present system uses no external regularization schedule. When the solution set $S$ is affine and the assumptions of
		Theorem~\ref{thm:anchor-induced-projection-selection} hold, it selects
		$P_S(z_0)$.
		\\[0.8ex]
		
		Anchor-acceleration and Halpern-type dynamics \eqref{FAH}
		&
		Feedback toward an externally prescribed fixed anchor.
		&
		Only $z_0$ is prescribed externally in the present system; thereafter, $z(t)$ evolves endogenously as a weighted memory of the past gradients.
		\\\\[0.8ex]
		
		Closed-loop damping \eqref{CL}
		&
		Nonlinear feedback acting through the velocity in the principal model.
		&
		The present feedback acts on the accelerated phase $y(t)$, with a nonlinear gain modulated by the squared coordinatewise displacement from the endogenous
		anchor.
		\\
		\bottomrule
	\end{tabularx}
\end{table}

\section{Existence and Uniqueness}\label{EU}

In this section, we establish the well-posedness of the system
\eqref{eq:system8}--\eqref{eq:z-dynamics}. Throughout this section, we assume that
$t_0>0$ and $ \lambda=\alpha-1> 0.$\\

\noindent To obtain the results in this section,  it is convenient to rewrite the second-order system as a first-order
nonautonomous system. Recall the phase variable
\begin{equation}\label{eq:phase-section4}
	y(t)=\lambda\bigl(x(t)-z(t)\bigr)+t\dot{x}(t).
\end{equation}
Differentiating it, we see that \eqref{eq:system7} is equivalent to the first-order
system 
\begin{equation}\label{eq:first-order-section4}
	\left\{
	\begin{aligned}
		\dot{x}(t)
		&=
		\frac{1}{t}
		\left[
		y(t)-\lambda\bigl(x(t)-z(t)\bigr)
		\right],
		\\[1.5mm]
		\dot{y}(t)
		&=
		(1-2a)t\nabla F(x(t))
		-\gamma y(t)
		-\frac{\eta}{t}D_M\bigl(x(t)-z(t)\bigr)y(t),
		\\[1.5mm]
		\dot{z}(t)
		&=
		-\frac{t}{\lambda}(1-a)\nabla F(x(t)).
	\end{aligned}
	\right.
\end{equation}

\begin{definition}\label{DF}
	Let $T>t_0$. A pair $(x,z)$ is called a strong solution of \eqref{eq:system7} on $[t_0,T]$ if
	$$
	x(t)\in C^2([t_0,T];\mathbb R^d),
	\qquad
	z(t)\in C^1([t_0,T];\mathbb R^d),
	$$
	and \eqref{eq:system8}--\eqref{eq:z-dynamics} hold for every $t\in[t_0,T]$. Equivalently, if $y(t)$ is defined by \eqref{eq:phase-section4}, then
	$$
	x(t),y(t),z(t)\in C^1([t_0,T];\mathbb R^d),
	$$
	and $(x(t),y(t),z(t))$ satisfies \eqref{eq:first-order-section4}.
\end{definition}

\subsection{Local Existence and Uniqueness}

\begin{theorem}\label{thm:local-existence-section4}
	Assume that $F\in C^1(\mathbb R^d)$ and that $\nabla F$ is locally Lipschitz on
	$\mathbb R^d$. Let $t_0>0$, $\lambda=\alpha-1>0$, and let
	$a,\gamma,\eta\in\mathbb R$. Then, for every initial condition
	$$
	x(t_0)=x_0,\qquad
	\dot{x}(t_0)=v_0,\qquad
	z(t_0)=z_0,
	$$
	the system \eqref{eq:system8}--\eqref{eq:z-dynamics} admits a unique local strong solution
	in the sense of Definition \ref{DF}.\\
	
	\noindent More precisely, there exist $T>t_0$ and a unique pair
	$$
	x(t)\in C^2([t_0,T];\mathbb R^d),
	\qquad
	z(t)\in C^1([t_0,T];\mathbb R^d),
	$$
	satisfying \eqref{eq:system8}--\eqref{eq:z-dynamics} and the prescribed initial condition.
	Moreover, there exists a maximal time
	$$
	T_{\max}\in(t_0,\infty]
	$$
	such that the solution exists uniquely on every compact interval contained in
	$[t_0,T_{\max})$. If $T_{\max}<\infty$, then
	\begin{equation}\label{eq:blowup-section4}
		\lim_{t\uparrow T_{\max}}
		\left(
		\|x(t)\|+\|\dot{x}(t)\|+\|z(t)\|
		\right)
		=
		\infty.
	\end{equation}
\end{theorem}
\noindent The proof is postponed to Appendix \ref{app:proofs-section4}.

\subsection{Global Existence and Uniqueness}
We now show that the local solution obtained in Theorem \ref{thm:local-existence-section4} does not blow up in finite time. The next theorem is for the regime $   \frac12<a<1.$

\begin{theorem}\label{thm:global-main-section4}
	Assume that $F\in C^1(\mathbb R^d)$ and that $\nabla F$ is locally Lipschitz on
	$\mathbb R^d$. Assume also that $F$ is convex and that     $\argmin F\neq\emptyset.$
	Let   $  t_0>0,~    \alpha\ge 3,~
	\lambda=\alpha-1,~
	\gamma\ge 0,~
	\eta\ge 0,$
	and assume that $
	\frac12<a<1.$
	Then, for every initial condition
	$$
	x(t_0)=x_0,\qquad
	\dot{x}(t_0)=v_0,\qquad
	z(t_0)=z_0,
	$$
	the system  \eqref{eq:system7} admits a unique global strong solution
	in the sense of Definition \ref{DF}. That is,
	$$
	x(t)\in C^2([t_0,\infty);\mathbb R^d),
	\qquad
	z(t)\in C^1([t_0,\infty);\mathbb R^d).
	$$
	Moreover, 
	$$
	y(t)\in C^1([t_0,\infty);\mathbb R^d).$$
\end{theorem}

\noindent The proof is postponed to Appendix \ref{app:proofs-section4}.

\hfill

\begin{remark}
	Next, we consider the endpoint, $a=1$. In this case, the anchor equation \eqref{eq:z-dynamics} becomes $    \dot{z}(t)=0.$
	Hence $z(t)=z_0$ for every $t\ge t_0$. This endpoint is therefore a fixed-anchor system,
	rather than an endogenous-anchor system.
\end{remark}

\begin{proposition}\label{prop:a-one-section4}
	Assume that $F\in C^1(\mathbb R^d)$, that $\nabla F$ is locally Lipschitz on
	$\mathbb R^d$, and that $F$ is bounded from below. Let $ t_0>0,~
	\alpha>1,~
	\lambda=\alpha-1>0,~
	\gamma\ge 0,~
	\eta\ge 0.$
	Let $a=1$. Then, for every initial condition
	$$
	x(t_0)=x_0,\qquad
	\dot{x}(t_0)=v_0,\qquad
	z(t_0)=z_0,
	$$
	the system \eqref{eq:system8}--\eqref{eq:z-dynamics} admits a unique global strong solution
	in the sense of Definition \ref{DF}.
\end{proposition}

\noindent The proof is postponed to Appendix \ref{app:proofs-section4}.

\begin{remark}
	Lastly, we consider the other critical endpoint $a=\frac{1}{2}$. Thus,  the coefficient $2a-1$ in the Lyapunov functional
	\eqref{eq:Ea-section5} vanishes. Therefore that functional no longer controls
	$z(t)$. Consequently, the proof of Theorem \ref{thm:global-main-section4} does not
	automatically extend to this endpoint. A simple global-existence result is nevertheless
	available under a linear-growth condition on $\nabla F$.
\end{remark}

\begin{proposition}\label{prop:a-half-section4}
	Assume that $F\in C^1(\mathbb R^d)$ and that $\nabla F$ is locally Lipschitz on
	$\mathbb R^d$. Assume moreover that there exist constants $c_0,c_1\ge 0$ such that
	\begin{equation}\label{eq:linear-growth-section4}
		\|\nabla F(\xi)\|
		\le
		c_0+c_1\|\xi\|,
		\qquad
		\xi\in\mathbb R^d.
	\end{equation}
	Let   $ t_0>0,~
	\lambda=\alpha-1>0,~
	\gamma\ge 0,~
	\eta\ge 0,$
	and let $a=1/2$. Then, for every initial condition
	$$
	x(t_0)=x_0,\qquad
	\dot{x}(t_0)=v_0,\qquad
	z(t_0)=z_0,
	$$
	the system \eqref{eq:system8}--\eqref{eq:z-dynamics} admits a unique global strong solution
	in the sense of Definition \ref{DF}.
\end{proposition}

\noindent The proof is postponed to Appendix \ref{app:proofs-section4}.

\section{Convergence, Selection, and Restoring-Damping Effects}
\label{sec:convergence-selection-damping}

\subsection{Lyapunov monotonicity} \label{subsec:lyapunov-monotonicity-section5}

\begin{proposition}[Accelerated Lyapunov monotonicity]
	\label{prop:accelerated-lyapunov-a-section5}
	
	Assume that $F$ is convex and continuously differentiable. Let $(x,z)$ be a strong solution of \eqref{eq:system7} on $[t_0,T]$, with $T>t_0$.
	
	\item[(i)] Assume that $    \frac12<a<1.$ Fix $x^\star\in\argmin F$.
	Define
	\begin{equation}\label{eq:Ea-section5}
		E_{a,x^\star}(t)
		:=
		(2a-1)t^2\bigl(F(x(t))-F^\star\bigr)
		+
		\frac12\|y(t)\|^2
		+
		\frac{\lambda^2(2a-1)}{2(1-a)}
		\|z(t)-x^\star\|^2.
	\end{equation}
	Then, for every $t\in[t_0,T]$ and for $  \alpha\geq 3,~
	\lambda=\alpha-1,~
	\gamma\geq 0, ~
	\eta\geq 0,$
	\begin{equation}\label{eq:Ea-drop-section5}
		\begin{aligned}
			\dot{E}_a(t)
			&=
			(2a-1)
			\left[
			2t\bigl(F(x(t))-F^\star\bigr)
			-
			\lambda t
			\bigl\langle x(t)-x^\star,\nabla F(x(t))\bigr\rangle
			\right]
			\\
			&\quad
			-
			\gamma\|y(t)\|^2
			-
			\frac{\eta}{t}
			\bigl\langle
			D_M\bigl(x(t)-z(t)\bigr)y(t),
			y(t)
			\bigr\rangle\leq 0.
		\end{aligned}
	\end{equation}
	In particular, $E_{a,x^\star}$ is nonincreasing on $[t_0,T]$.
	
	\item[(ii)] Assume that $ a=\frac12.$ Define
	\begin{equation}\label{eq:Ehalf-section5}
		E_{\frac12}(t)
		:=
		\frac12\|y(t)\|^2.
	\end{equation}
	Then, for every $t\in[t_0,T]$, and for $  
	\gamma\geq 0, ~
	\eta\geq 0,$
	\begin{equation}\label{eq:Ehalf-drop-section5}
		\dot{E}_{\frac12}(t)
		=
		-
		\gamma\|y(t)\|^2
		-
		\frac{\eta}{t}
		\bigl\langle
		D_M\bigl(x(t)-z(t)\bigr)y(t),
		y(t)
		\bigr\rangle
		\leq 0.
	\end{equation}
	In particular, $E_{\frac12}$ is nonincreasing on $[t_0,T]$.
	
	\item[(iii)] Fix $a=1$. Then, the anchor is fixed: $z(t)\equiv z_0.$ Assume further that $z_0=x^\star ~~
	\text{for some }x^\star\in\operatorname{argmin}F.
	$ ($ z(t)\equiv z_0=x^\star$). 
	Define
	\begin{equation}\label{eq:Eone-section5}
		E_{1,x^\star}(t)
		:=
		t^2\bigl(F(x(t))-F^\star\bigr)
		+
		\frac12
		\left\|
		\lambda\bigl(x(t)-x^\star\bigr)+t\dot{x}(t)
		\right\|^2.
	\end{equation}
	Then, for every $t\in[t_0,T]$, and for $  \alpha\geq 3,~
	\lambda=\alpha-1,~
	\gamma\geq 0, ~
	\eta\geq 0,$
	\begin{equation}\label{eq:Eone-drop-section5}
		\begin{aligned}
			\dot{E}_1(t)
			&=
			2t\bigl(F(x(t))-F^\star\bigr)
			-
			\lambda t
			\bigl\langle
			x(t)-x^\star,\nabla F(x(t))
			\bigr\rangle
			-
			\gamma
			\bigl\|
			\lambda(x(t)-x^\star)+t\dot{x}(t)
			\bigr\|^2
			\\
			&\quad
			-
			\frac{\eta}{t}
			\bigl\langle
			D_M(x(t)-x^\star)
			\bigl(\lambda(x(t)-x^\star)+t\dot{x}(t)\bigr),
			\\
			&\qquad\qquad
			\lambda(x(t)-x^\star)+t\dot{x}(t)
			\bigr\rangle
			\leq0.
		\end{aligned}
	\end{equation}
	In particular, $E_{1,x^\star}$ is nonincreasing on $[t_0,T]$.
\end{proposition}

\begin{proof}
	\noindent
	\textbf{Proof of (i).}
	Differentiating the first term in
	\eqref{eq:Ea-section5}, we obtain
	$$
	\begin{aligned}
		\frac{d}{dt}
		\left[
		(2a-1)t^2\bigl(F(x(t))-F^\star\bigr)
		\right]
		&=
		(2a-1)
		\left[
		2t\bigl(F(x(t))-F^\star\bigr)
		+
		t^2
		\bigl\langle
		\nabla F(x(t)),\dot{x}(t)
		\bigr\rangle
		\right].
	\end{aligned}
	$$
	Differentiating the second term in
	\eqref{eq:Ea-section5} and using the phase identity \eqref{eq:system7},
	$$
	\begin{aligned}
		\frac{d}{dt}\frac12\|y(t)\|^2
		&=
		\bigl\langle y(t),\dot y(t)\bigr\rangle
		\\
		&=
		(1-2a)t
		\bigl\langle y(t),\nabla F(x(t))\bigr\rangle
		-
		\gamma\|y(t)\|^2
		-
		\frac{\eta}{t}
		\bigl\langle
		D_M\bigl(x(t)-z(t)\bigr)y(t),
		y(t)
		\bigr\rangle.
	\end{aligned}
	$$
	\noindent Since $    y(t)=\lambda\bigl(x(t)-z(t)\bigr)+t\dot{x}(t),$
	we have
	\begin{eqnarray}\label{5.12}
		\begin{aligned}
			\frac{d}{dt}\frac12\|y(t)\|^2
			&=
			-(2a-1)\lambda t
			\bigl\langle x(t)-z(t),\nabla F(x(t))\bigr\rangle-
			(2a-1)t^2
			\bigl\langle
			\dot{x}(t),\nabla F(x(t))
			\bigr\rangle
			\\
			&\quad
			-
			\gamma\|y(t)\|^2
			-
			\frac{\eta}{t}
			\bigl\langle
			D_M\bigl(x(t)-z(t)\bigr)y(t),
			y(t)
			\bigr\rangle.
		\end{aligned}
	\end{eqnarray}
	Differentiating the last term in
	\eqref{eq:Ea-section5}, and using the anchor equation \eqref{eq:z-dynamics}, we have
	$$
	\begin{aligned}
		\frac{d}{dt}
		\left[
		\frac{\lambda^2(2a-1)}{2(1-a)}
		\|z(t)-x^\star\|^2
		\right]
		&=
		\frac{\lambda^2(2a-1)}{1-a}
		\bigl\langle z(t)-x^\star,\dot z(t)\bigr\rangle
		\\
		&=
		-\lambda(2a-1)t
		\bigl\langle
		z(t)-x^\star,\nabla F(x(t))
		\bigr\rangle.
	\end{aligned}
	$$
	Adding these three derivatives gives
	\begin{equation}\label{eq:lyapunov-derivative-exact-section4}
		\begin{aligned}
			\dot{E}_a(t)
			&=
			(2a-1)
			\left[
			2t\bigl(F(x(t))-F^\star\bigr)
			-
			\lambda t
			\bigl\langle x(t)-x^\star,\nabla F(x(t))\bigr\rangle
			\right]
			\\
			&\quad
			-
			\gamma\|y(t)\|^2
			-
			\frac{\eta}{t}
			\bigl\langle
			D_M\bigl(x(t)-z(t)\bigr)y(t),
			y(t)
			\bigr\rangle.
		\end{aligned}
	\end{equation}
	
	 Since $F$ is convex and $x^\star\in\argmin F$, we have
	$$
	\bigl\langle x(t)-x^\star,\nabla F(x(t))\bigr\rangle
	\geq
	F(x(t))-F^\star.
	$$
	Moreover, by \eqref{eq:DM-definition},
	$$
	D_M\bigl(x(t)-z(t)\bigr)\succeq0.
	$$
	Thus
	$$
	\bigl\langle
	D_M\bigl(x(t)-z(t)\bigr)y(t),
	y(t)
	\bigr\rangle
	\geq0.
	$$
	Using these facts in \eqref{eq:lyapunov-derivative-exact-section4}, we obtain
	\begin{equation}\label{eq:lyapunov-derivative-ineq-section4}
		\begin{aligned}
			\dot{E}_a(t)
			&\le
			(2a-1)(2-\lambda)t\bigl(F(x(t))-F^\star\bigr)
			-
			\gamma\|y(t)\|^2.
		\end{aligned}
	\end{equation}
	Since $\lambda=\alpha-1\geq2$, $\gamma\geq0$, $\eta\geq0$, and
	$2a-1>0$, it follows that
	\begin{equation*}
		\dot{E}_a(t)\le 0.
	\end{equation*}
	
	\smallskip
	
	\noindent
	\textbf{Proof of (ii).} 
	When $a=\frac12$, the second term in \eqref{eq:first-order-section4}  becomes
	$$
	\dot{y}(t)
	=
	-\gamma y(t)
	-
	\frac{\eta}{t}D_M(x(t)-z(t))y(t).
	$$
	Therefore,
	$$
	\begin{aligned}
		\dot{E}_{\frac12}(t)
		&=
		\langle y(t),\dot{y}(t)\rangle=
		-\gamma\|y(t)\|^2
		-
		\frac{\eta}{t}
		\bigl\langle
		D_M(x(t)-z(t))y(t),
		y(t)
		\bigr\rangle.
	\end{aligned}
	$$
	Since $D_M(x(t)-z(t))\succeq0$, the right-hand side is nonpositive.
	
	\smallskip
	
	\noindent \textbf{Proof of (iii).} 
	The proof is very similar to the proof of (i).
\end{proof}

\begin{remark}\label{rem:a1-matched-anchor-section5}
	The condition $z(t)\equiv z_0=x^\star$ in Proposition 
	\ref{prop:accelerated-lyapunov-a-section5}(iii) is not needed for well-posedness
	or for dissipativity of the fixed-anchor dynamics. It is needed only to
	recover the classical accelerated Lyapunov functional centered at a
	minimizer. For arbitrary fixed anchor $z_0$, the endpoint $a=1$ admits
	instead the mechanical Lyapunov functional stated in Proposition
	\ref{prop:fixed-anchor-mechanical-section5} below.
\end{remark}

\begin{proposition}[Mechanical Lyapunov monotonicity for arbitrary fixed anchor]
	\label{prop:fixed-anchor-mechanical-section5}
	Assume that $F\in C^1(\mathbb R^d)$ is bounded from below and let
	$    \alpha>1,~
	\lambda=\alpha-1,~
	\gamma\geq0$ and $
	\eta\geq0.$
	Fix $a=1$, and let $(x,z)$ be a strong solution of
	\eqref{eq:system8}--\eqref{eq:z-dynamics} on $[t_0,T]$, with $T>t_0$.
	Then  $z(t)\equiv z_0$
	for some arbitrary fixed anchor $z_0\in\mathbb R^d$.\\
	\noindent Set    $ r(t):=x(t)-z_0,\qquad
	\Phi_M(r)
	:=
	\sum_{i=1}^d m_i^2r_i^4
	=
	\bigl\langle D_M(r)r,r\bigr\rangle,
	\qquad
	F_{\inf}:=\inf_{\xi\in\mathbb R^d}F(\xi).$
	Define the mechanical energy
	\begin{equation}\label{eq:mechanical-energy-section5}
		\mathcal H(t)
		:=
		\frac12\|\dot{x}(t)\|^2
		+
		F(x(t))-F_{\inf}
		+
		\frac{\gamma\lambda}{2t}\|r(t)\|^2
		+
		\frac{\eta\lambda}{4t^2}\Phi_M(r(t)).
	\end{equation}
	Then, for every $t\in[t_0,T]$,
	\begin{equation}\label{eq:mechanical-energy-identity-section5}
		\begin{aligned}
			\dot{\mathcal H}(t)
			&=
			-
			\left(\frac{\alpha}{t}+\gamma\right)
			\|\dot{x}(t)\|^2
			-
			\frac{\eta}{t}
			\bigl\langle
			D_M(r(t))\dot{x}(t),
			\dot{x}(t)
			\bigr\rangle
			-
			\frac{\gamma\lambda}{2t^2}\|r(t)\|^2
			-
			\frac{\eta\lambda}{2t^3}\Phi_M(r(t))\leq 0.
		\end{aligned}
	\end{equation}
	Consequently, $\mathcal H$ is nonincreasing on $[t_0,T]$.
\end{proposition}

\begin{proof}
	Since $a=1$, the anchor equation \eqref{eq:z-dynamics} gives $z(t)\equiv z_0.$  Writing
	$$
	r(t)=x(t)-z_0,
	$$
	the second-order equation \eqref{eq:system8} becomes
	\begin{equation}\label{eq:a1-fixed-anchor-system-section5}
		\begin{aligned}
			\ddot{x}(t)
			&+
			\left(\frac{\alpha}{t}+\gamma\right)\dot{x}(t)
			+
			\frac{\gamma\lambda}{t}r(t)
			+
			\frac{\eta}{t}D_M(r(t))\dot{x}(t)+
			\frac{\eta\lambda}{t^2}D_M(r(t))r(t)
			+
			\nabla F(x(t))
			=
			0.
		\end{aligned}
	\end{equation}
	We differentiate each term in \eqref{eq:mechanical-energy-section5}. First,
	$$
	\frac{d}{dt}
	\left[
	\frac12\|\dot{x}(t)\|^2
	+
	F(x(t))-F_{\inf}
	\right]
	=
	\bigl\langle \dot{x}(t),\ddot{x}(t)\bigr\rangle
	+
	\bigl\langle \nabla F(x(t)),\dot{x}(t)\bigr\rangle.
	$$
	Next,
	$$
	\begin{aligned}
		\frac{d}{dt}
		\left[
		\frac{\gamma\lambda}{2t}\|r(t)\|^2
		\right]
		&=
		\frac{\gamma\lambda}{t}
		\bigl\langle r(t),\dot{x}(t)\bigr\rangle
		-
		\frac{\gamma\lambda}{2t^2}\|r(t)\|^2.
	\end{aligned}
	$$
	Since $    \nabla\Phi_M(r)=4D_M(r)r,$
	we have
	$$
	\frac{d}{dt}\Phi_M(r(t))
	=
	4
	\bigl\langle
	D_M(r(t))r(t),
	\dot{x}(t)
	\bigr\rangle.
	$$
	Therefore,
	$$
	\begin{aligned}
		\frac{d}{dt}
		\left[
		\frac{\eta\lambda}{4t^2}\Phi_M(r(t))
		\right]
		&=
		\frac{\eta\lambda}{t^2}
		\bigl\langle
		D_M(r(t))r(t),
		\dot{x}(t)
		\bigr\rangle
		-
		\frac{\eta\lambda}{2t^3}\Phi_M(r(t)).
	\end{aligned}
	$$
	\noindent Using \eqref{eq:a1-fixed-anchor-system-section5}, we get
	$$
	\begin{aligned}
		\bigl\langle \dot{x}(t),\ddot{x}(t)\bigr\rangle
		&=
		-
		\left(\frac{\alpha}{t}+\gamma\right)
		\|\dot{x}(t)\|^2
		-
		\frac{\gamma\lambda}{t}
		\bigl\langle r(t),\dot{x}(t)\bigr\rangle
		-
		\frac{\eta}{t}
		\bigl\langle
		D_M(r(t))\dot{x}(t),
		\dot{x}(t)
		\bigr\rangle
		\\
		&\quad
		-
		\frac{\eta\lambda}{t^2}
		\bigl\langle
		D_M(r(t))r(t),
		\dot{x}(t)
		\bigr\rangle
		-
		\bigl\langle
		\nabla F(x(t)),\dot{x}(t)
		\bigr\rangle.
	\end{aligned}
	$$
	Adding the differentiated terms, we get
	$$
	\begin{aligned}
		\dot{\mathcal H}(t)
		&=
		-
		\left(\frac{\alpha}{t}+\gamma\right)
		\|\dot{x}(t)\|^2
		-
		\frac{\eta}{t}
		\bigl\langle
		D_M(r(t))\dot{x}(t),
		\dot{x}(t)
		\bigr\rangle
		-
		\frac{\gamma\lambda}{2t^2}\|r(t)\|^2
		-
		\frac{\eta\lambda}{2t^3}\Phi_M(r(t)).
	\end{aligned}
	$$
	This completes the proof.
\end{proof}

\begin{remark}\label{rem:mechanical-not-rate-section5}
	Proposition \ref{prop:fixed-anchor-mechanical-section5} gives a
	Lyapunov monotonicity statement for arbitrary fixed anchor $z_0$. However,
	it is a mechanical-energy monotonicity result; it shows that the kinetic energy, objective energy, and the anchor-restoring potentials dissipate monotonically. It does not directly yield the
	accelerated objective-value estimate: $   F(x(t))-F^\star=\mathcal O(t^{-2}).$
	Such an estimate will be studied in the next subsection.
\end{remark}

\subsection{Accelerated Objective-value estimates}
\label{subsec:objective-value-estimates-section5}

In this subsection we derive explicit estimates for the decay of the objective-value
residual $    F(x(t))-F^\star,$
along the global strong solutions of \eqref{eq:system8}--\eqref{eq:z-dynamics}.
The estimates follow directly from the Lyapunov inequalities established in
Subsection \ref{subsec:lyapunov-monotonicity-section5}. We emphasize that the nonlinear restoring--damping term
does not deteriorate the accelerated value rate. Rather, it appears with a favorable
sign in the Lyapunov dissipation and therefore strengthens the decay mechanism.

\begin{theorem}
	\label{thm:value-estimate-a-section5}
	Assume that $F\in C^1(\mathbb R^d)$ and that $\nabla F$ is locally Lipschitz on
	$\mathbb R^d$. Assume also that $F$ is convex and that $    \argmin F\neq\emptyset.$
	Suppose that $     \alpha\geq3,
	~
	\lambda=\alpha-1\geq2,
	~
	\gamma\geq0,
	~
	\eta\geq0.$ 
	Fix $    \frac12<a<1,$ and $x^\star\in\argmin F,$
	and let $(x,z)$ be the corresponding global strong solution. 
	Then, for every $t\geq t_0$,
	\begin{equation}
		\label{eq:value-bound-basic-section5}
		F(x(t))-F^\star
		\leq
		\frac{E_{a,x^\star}(t)}{(2a-1)t^2}
		\leq
		\frac{E_{a,x^\star}(t_0)}{(2a-1)t^2}.
	\end{equation}
	Consequently,
	\begin{equation}
		\label{eq:value-rate-O-section5}
		F(x(t))-F^\star
		=
		\mathcal{O}(t^{-2})
		\qquad
		\text{as }t\to\infty.
	\end{equation}
	More explicitly, if $   x_0:=x(t_0), 
	~
	v_0:=\dot{x}(t_0),
	~
	z_0:=z(t_0),$
	and $     y_0:=\lambda(x_0-z_0)+t_0v_0,$
	then
	\begin{equation}
		\label{eq:explicit-value-bound-section5}
		\begin{aligned}
			F(x(t))-F^\star
			\leq
			\frac{1}{t^2}
			\Bigg[
			&t_0^2\bigl(F(x_0)-F^\star\bigr)
			+
			\frac{\|y_0\|^2}{2(2a-1)}
			+
			\frac{\lambda^2}{2(1-a)}
			\|z_0-x^\star\|^2
			\Bigg],
			\qquad t\geq t_0.
		\end{aligned}
	\end{equation}
\end{theorem}

\begin{proof}
	By definition \eqref{eq:Ea-section5}, 
	$$
	F(x(t))-F^\star
	\leq
	\frac{E_{a,x^\star}(t)}{(2a-1)t^2}.
	$$
	By Proposition \ref{prop:accelerated-lyapunov-a-section5} (i), the functional
	$E_{a,x^\star}$ is nonincreasing. Therefore
	$$
	F(x(t))-F^\star
	\leq
	\frac{E_{a,x^\star}(t)}{(2a-1)t^2}
	\leq
	\frac{E_{a,x^\star}(t_0)}{(2a-1)t^2} \qquad
	t\geq t_0.
	$$
	Finally,
	$$
	\begin{aligned}
		E_{a,x^\star}(t_0)
		&=
		(2a-1)t_0^2\bigl(F(x_0)-F^\star\bigr)
		+
		\frac12\|y_0\|^2
		+
		\frac{\lambda^2(2a-1)}{2(1-a)}
		\|z_0-x^\star\|^2.
	\end{aligned}
	$$
	Dividing by $2a-1$ gives the explicit estimate.
\end{proof}

\begin{corollary}[Best constant obtained from the Lyapunov estimate]
	\label{cor:best-constant-section5}
	Under the assumptions of Theorem \ref{thm:value-estimate-a-section5}, define
	$$
	d_0:=\operatorname{dist}(z_0,\operatorname{argmin}F).
	$$
	Then, for every $t\geq t_0$,
	\begin{equation}
		\label{eq:best-constant-bound-section5}
		F(x(t))-F^\star
		\leq
		\frac{C_a}{t^2},
	\end{equation}
	where
	\begin{equation}
		\label{eq:Ca-section5}
		C_a
		:=
		t_0^2\bigl(F(x_0)-F^\star\bigr)
		+
		\frac{\|y_0\|^2}{2(2a-1)}
		+
		\frac{\lambda^2}{2(1-a)}d_0^2.
	\end{equation}
\end{corollary}

\begin{proof}
	The estimate follows from Theorem \ref{thm:value-estimate-a-section5} by
	minimizing the right-hand side of the explicit bound over all
	$$
	x^\star\in\operatorname{argmin}F.
	$$
	Since $F$ is convex and continuous, the set $\operatorname{argmin}F$ is closed
	and convex. Hence
	$$
	\inf_{x^\star\in\operatorname{argmin}F}
	\|z_0-x^\star\|^2
	=
	\operatorname{dist}(z_0,\operatorname{argmin}F)^2.
	$$
	Therefore,
	$$
	F(x(t))-F^\star
	\leq
	\frac{C_a}{t^2}.
	$$
\end{proof}

\begin{proposition}[Optimization of the explicit Lyapunov bound with respect to $a$]
	\label{prop:optimize-a-section5}
	Fix the initial data $x_0,v_0,z_0$, and fix all parameters except $a$. For each $
	a\in\left(\frac12,1\right),$
	let $(x_a,z_a)$ denote the global strong solution corresponding to this value of
	$a$, and let $C_a$ be the explicit Lyapunov-bound (upper-bound) constant \eqref{eq:Ca-section5}.
	
	\noindent Set $$ Y_0:=\|y_0\| ~~\mbox{and}~~ 
	s:=2a-1\in(0,1).$$
	
	\noindent If
	$$
	Y_0>0
	\qquad
	\text{and}
	\qquad
	d_0>0,
	$$
	then this explicit upper-bound constant has a unique minimizer
	\begin{equation}
		\label{eq:sstar-section5}
		s_\star
		=
		\frac{Y_0}{Y_0+\sqrt2\,\lambda d_0},
	\end{equation}
	or equivalently 
	\begin{equation}
		\label{eq:sstar-section5a}
		a_\star
		=
		\frac{1}{2}\left(1+\frac{Y_0}{Y_0+\sqrt2\,\lambda d_0}\right).
	\end{equation}
	At this value,
	\begin{equation}
		\label{eq:Cstar-section5}
		C_{a_\star}
		=
		t_0^2\bigl(F(x_0)-F^\star\bigr)
		+
		\left(
		\frac{Y_0}{\sqrt2}
		+
		\lambda d_0
		\right)^2.
	\end{equation}
\end{proposition}

\begin{proof}
	Since $    s=2a-1, $
	we have $
	\frac{\lambda^2}{2(1-a)}d_0^2
	=
	\frac{\lambda^2d_0^2}{1-s}.$
	Thus
	$$
	C_a
	=
	t_0^2\bigl(F(x_0)-F^\star\bigr)
	+
	\frac{Y_0^2}{2s}
	+
	\frac{\lambda^2d_0^2}{1-s}.
	$$
	The $a$-dependent part is
	$$
	\psi(s)
	:=
	\frac{Y_0^2}{2s}
	+
	\frac{\lambda^2d_0^2}{1-s},
	\qquad
	s\in(0,1).
	$$
	Differentiating,
	$$
	\psi'(s)
	=
	-\frac{Y_0^2}{2s^2}
	+
	\frac{\lambda^2d_0^2}{(1-s)^2}.
	$$
	The critical-point condition $\psi'(s)=0$ is equivalent to
	$$
	\frac{Y_0}{\sqrt2\,s}
	=
	\frac{\lambda d_0}{1-s}.
	$$
	Solving for $s$ gives
	$$
	s_\star
	=
	\frac{Y_0}{Y_0+\sqrt2\,\lambda d_0}.
	$$
	Moreover,
	$$
	\psi''(s)
	=
	\frac{Y_0^2}{s^3}
	+
	\frac{2\lambda^2d_0^2}{(1-s)^3}
	>
	0.
	$$
	Hence $\psi$ is strictly convex, and $s_\star$ is the unique minimizer.
	Substituting $s_\star$ into $\psi$, we get
	$$
	\min_{s\in(0,1)}\psi(s)
	=
	\left(
	\frac{Y_0}{\sqrt2}
	+
	\lambda d_0
	\right)^2.
	$$
	This proves the result.
\end{proof}

\begin{remark}[Interpretation and limitation of the parameter optimization]
	\label{rem:parameter-optimization-limitation-section5}
	Proposition \ref{prop:optimize-a-section5} shows how the splitting parameter
	$a$ balances the initial phase mismatch and the initial anchor mismatch in the
	explicit Lyapunov upper bound.
	The term $
	\frac{\|y_0\|^2}{2(2a-1)} $
	penalizes choosing $a$ too close to $\frac12$, while the term $
	\frac{\lambda^2}{2(1-a)}
	\operatorname{dist}(z_0,\operatorname{argmin}F)^2$ 
	penalizes choosing $a$ too close to $1$ when the initial anchor is not already
	close to the solution set. Thus, at the level of this Lyapunov estimate, smaller values of $a$ favor memory
	formation through the anchor, whereas larger values of $a$ favor direct inertial
	descent. The value $a_\star$ in \eqref{eq:sstar-section5a} gives the optimal balance for the explicit Lyapunov
	upper bound. This optimization concerns only the upper bound obtained from the Lyapunov estimate
	across the family of $a$-dependent trajectories. It does not necessarily identify
	the parameter value that minimizes the actual objective-value trajectory.
\end{remark}

\begin{remark}[Degenerate cases in the optimization over $a$]
	\label{rem:degenerate-a-optimization-section5}
	The interior optimizer in Proposition \ref{prop:optimize-a-section5} assumes
	$    Y_0>0~ 
	\text{and} ~
	d_0>0.
	$
	
	\noindent If
	$$
	Y_0=0
	\qquad
	\text{and}
	\qquad
	d_0>0,
	$$
	then the infimum of the bound is approached as
	$$
	s\downarrow0,
	\qquad
	a\downarrow\frac12,
	$$
	but it is not attained in the open interval
	$    a\in\left(\frac12,1\right).$
	
	\noindent If
	$$
	Y_0>0
	\qquad
	\text{and}
	\qquad
	d_0=0,
	$$
	then the infimum is approached as
	$$
	s\uparrow1,
	\qquad
	a\uparrow1.
	$$
	In this case $z_0\in\operatorname{argmin}F$, so the limiting endpoint $a=1$
	is compatible with the matched-anchor estimate by choosing $x^\star=z_0$.
	
	\noindent If
	$$
	Y_0=d_0=0,
	$$
	then the constant $C_a$ is independent of $a$ and equals $
	t_0^2\bigl(F(x_0)-F^\star\bigr).$
\end{remark}

\begin{theorem}[Matched-anchor endpoint $a=1$]
	\label{thm:value-estimate-a-one-section5}
	Assume that $F\in C^1(\mathbb R^d)$ and that $\nabla F$ is locally Lipschitz on
	$\mathbb R^d$. Assume also that $F$ is convex and that $    \argmin F\neq\emptyset.$
	Suppose that $     \alpha\geq3,
	~
	\lambda=\alpha-1\geq2,
	~
	\gamma\geq0,
	~
	\eta\geq0.$ 
	Fix $    a=1,$ 
	and let $(x,z)$ be the corresponding global strong solution. 
	Assume further that $
	z(t)\equiv z_0=x^\star
	~
	\text{for some }x^\star\in\operatorname{argmin}F.$  
	Then, for every $t\geq t_0$,
	\begin{equation}
		\label{eq:a-one-value-bound-section5}
		F(x(t))-F^\star
		\leq
		\frac{E_{1,x^\star}(t)}{t^2}
		\leq
		\frac{E_{1,x^\star}(t_0)}{t^2}.
	\end{equation}
	Equivalently,
	\begin{equation}
		\label{eq:a-one-explicit-section5}
		F(x(t))-F^\star
		\leq
		\frac{1}{t^2}
		\left[
		t_0^2\bigl(F(x_0)-F^\star\bigr)
		+
		\frac12
		\bigl\|
		\lambda(x_0-x^\star)+t_0v_0
		\bigr\|^2
		\right].
	\end{equation}
	In particular,
	$$
	F(x(t))-F^\star=\mathcal O(t^{-2})
	\qquad
	\text{as }t\to\infty.
	$$
\end{theorem}

\begin{proof}
	It follows the same argument as the proof of Theorem \ref{thm:value-estimate-a-section5}.
\end{proof}

\begin{proposition}[Integrated objective-value and phase dissipation]
	\label{prop:integrated-value-dissipation-section5}
	Assume the hypotheses of Theorem \ref{thm:value-estimate-a-section5}. Then, for
	every $t\geq t_0$,
	\begin{equation}
		\label{eq:integrated-value-dissipation-section5}
		\begin{aligned}
			&(2a-1)(\lambda-2)
			\int_{t_0}^{t}
			s\bigl(F(x(s))-F^\star\bigr)\,ds
			+
			\gamma
			\int_{t_0}^{t}
			\|y(s)\|^2\,ds
			\\
			&\quad
			+
			\eta
			\int_{t_0}^{t}
			\frac{1}{s}
			\bigl\langle
			D_M(x(s)-z(s))y(s),
			y(s)
			\bigr\rangle
			ds
			\leq
			E_{a,x^\star}(t_0)-E_{a,x^\star}(t)
			\leq
			E_{a,x^\star}(t_0).
		\end{aligned}
	\end{equation}
	In particular, if  $   \lambda>2,$ 
	then
	$$
	\int_{t_0}^{\infty}
	s\bigl(F(x(s))-F^\star\bigr)\,ds
	<
	\infty.
	$$
	Moreover, if $\gamma>0$, then
	$$
	\int_{t_0}^{\infty}\|y(s)\|^2\,ds
	<
	\infty,
	$$
	and if $\eta>0$, then
	$$
	\int_{t_0}^{\infty}
	\frac{1}{s}
	\bigl\langle
	D_M(x(s)-z(s))y(s),
	y(s)
	\bigr\rangle
	ds
	<
	\infty.
	$$
\end{proposition}

\begin{proof}
	By Proposition \ref{prop:accelerated-lyapunov-a-section5},
	$$
	\begin{aligned}
		\dot{E}_a(t)
		&\leq
		-(2a-1)(\lambda-2)t\bigl(F(x(t))-F^\star\bigr)
		-
		\gamma\|y(t)\|^2
		-
		\frac{\eta}{t}
		\bigl\langle
		D_M(x(t)-z(t))y(t),
		y(t)
		\bigr\rangle.
	\end{aligned}
	$$
	Integrating over $[t_0,t]$, we get
	$$
	\begin{aligned}
		E_{a,x^\star}(t)-E_{a,x^\star}(t_0)
		&\leq
		-(2a-1)(\lambda-2)
		\int_{t_0}^{t}
		s\bigl(F(x(s))-F^\star\bigr)\,ds
		\\
		&\quad
		-
		\gamma
		\int_{t_0}^{t}
		\|y(s)\|^2\,ds-
		\eta
		\int_{t_0}^{t}
		\frac{1}{s}
		\bigl\langle
		D_M(x(s)-z(s))y(s),
		y(s)
		\bigr\rangle
		ds.
	\end{aligned}
	$$
	Rearranging gives \eqref{eq:integrated-value-dissipation-section5}. Since
	$E_{a,x^\star}(t)\geq0$, we also have
	$$
	E_{a,x^\star}(t_0)-E_{a,x^\star}(t)
	\leq
	E_{a,x^\star}(t_0).
	$$
	Letting $t\to\infty$ and using the nonnegativity of all three integrands gives
	the stated integrability conclusions. The first conclusion requires $\lambda>2$,
	since otherwise the coefficient $\lambda-2$ vanishes.
\end{proof}

\begin{remark}[Role of the restoring--damping coefficient]
	\label{rem:eta-role-section5}
	The coefficient $\eta$ does not appear in the leading explicit constant in the
	objective-value estimate because that bound is obtained by discarding all
	dissipative terms and using only the initial Lyapunov energy. Its effect is nevertheless visible in the integrated inequality, where the
	restoring--damping mechanism contributes the additional nonnegative dissipation
	term
	$$
	\eta
	\int_{t_0}^{t}
	\frac{1}{s}
	\bigl\langle
	D_M(x(s)-z(s))y(s),
	y(s)
	\bigr\rangle
	ds.
	$$
	Thus the restoring--damping mechanism provides additional dissipation in
	the Lyapunov balance while preserving the accelerated worst-case objective-value
	rate
	$$
	F(x(t))-F^\star=\mathcal O(t^{-2}).
	$$
\end{remark}

\hfill

\begin{remark}[The endpoint $a=\frac12$]
	\label{rem:endpoint-half-section5}
	When $a=\frac12,$
	the coefficient $2a-1$ vanishes. The Lyapunov functional $
	E_{\frac12}(t)
	=
	\frac12\|y(t)\|^2 $
	controls only the phase variable and does not contain the objective-value term
	$
	t^2\bigl(F(x(t))-F^\star\bigr).$
	Therefore the preceding argument does not yield an accelerated objective-value
	estimate at the endpoint $a=\frac12$. Additional assumptions or a different
	Lyapunov functional may be required to obtain a value-rate statement in this
	critical regime. This critical case remains open.
\end{remark}

\subsection{Convergence of Trajectories}
\label{subsec:point-convergence-strict-phase-feedback}

We now prove point convergence of the primal trajectory $x(t)$.    A key point is
that the Lyapunov functional controls the anchor distance
\(\|z(t)-x^\star\|^2\), rather than the primal distance
\(\|x(t)-x^\star\|^2\). {\it We therefore first obtain asymptotic information on
	the anchor \(z(t)\), and then transfer this information to \(x(t)\)  through an anchor-to-primal transfer argument based on the 
	phase equation \eqref{eq:phase-section4}}. 

\begin{lemma}[Boundedness of the trajectories]
	\label{lem:boundedness-trajectories-recentered-phase}
	Assume that $F:\mathbb R^d\to\mathbb R$ is convex and continuously differentiable. Assume that $
	t_0>0, ~\frac12<a<1,~ \alpha\ge3,$ and $
	\lambda=\alpha-1\ge2,~\gamma\ge0,~\eta\ge0.$        
	Let \(x^\star\in\arg\min F\), and let \((x,y,z)\) be a  global strong solution of \eqref{eq:system7}.
	Then $        x(t),~ \dot x(t),~ y(t)$ and $z(t)$ are bounded on \([t_0,\infty)\).
\end{lemma}

\begin{proof}
	Set $        \mu:=2a-1$
	and define the Lyapunov function
	$$
	E_{a,x^\star}(t)
	=
	(2a-1)t^2\bigl(F(x(t))-F^\star\bigr)
	+
	\frac12\|y(t)\|^2
	+
	\frac{\lambda^2(2a-1)}{2(1-a)}
	\|z(t)-x^\star\|^2.
	$$
	By Proposition \ref{prop:accelerated-lyapunov-a-section5}, 
	$$
	E_{a,x^\star}(t)\le E_{a,x^\star}(t_0),
	\qquad t\ge t_0.
	$$
	Consequently,
	$$
	\frac12\|y(t)\|^2
	\le
	E_{a,x^\star}(t_0),
	~ t\ge t_0, ~~\mbox{and}~~
	\frac{\lambda^2\mu}{2(1-a)}
	\|z(t)-x^\star\|^2
	\le
	E_{a,x^\star}(t_0),
	~ t\ge t_0.
	$$
	Hence \(y(t)\) and \(z(t)\) are bounded on \([t_0,\infty)\). 
	
	\noindent  Define the recentered phase variable
	$$
	w(t):=y(t)+\lambda(z(t)-x^\star).
	$$
	Using the phase relation, we obtain
	$$
	\begin{aligned}
		w(t)
		&=
		\lambda(x(t)-z(t))+t\dot x(t)
		+\lambda(z(t)-x^\star)  \\
		&=
		\lambda(x(t)-x^\star)+t\dot x(t).
	\end{aligned}
	$$
	Since \(y(t)\) and \(z(t)-x^\star\) are bounded, \(w(t)\) is bounded. Hence there exists a constant \(K\ge0\) such that
	$$
	\|w(t)\|^2\le K,
	\qquad t\ge t_0.
	$$
	Define
	$$
	h(t):=\|x(t)-x^\star\|^2.
	$$
	We compute
	$$
	\begin{aligned}
		\|w(t)\|^2
		&=
		\lambda^2\|x(t)-x^\star\|^2
		+
		2\lambda t\langle x(t)-x^\star,\dot x(t)\rangle
		+
		t^2\|\dot x(t)\|^2        \\
		&=
		\lambda^2 h(t)+\lambda t\dot h(t)+t^2\|\dot x(t)\|^2.
	\end{aligned}
	$$
	Since \(t^2\|\dot x(t)\|^2\ge0\), it follows that
	$$
	\lambda^2 h(t)+\lambda t\dot h(t)
	\le
	\|w(t)\|^2
	\le K.
	$$
	Dividing by \(\lambda>0\) and \(t\ge t_0>0\), we get
	$$
	\dot h(t)+\frac{\lambda}{t}h(t)
	\le
	\frac{K}{\lambda t}.
	$$
	Multiplying by the integrating factor \(t^\lambda\) and integrating from \(t_0\) to \(t\), we find
	$$
	h(t)
	\le
	\left(\frac{t_0}{t}\right)^\lambda h(t_0)
	+
	\frac{K}{\lambda^2}
	\left(
	1-\left(\frac{t_0}{t}\right)^\lambda
	\right).
	$$
	Therefore
	$$
	h(t)
	\le
	\max\left\{
	h(t_0),
	\frac{K}{\lambda^2}
	\right\},
	\qquad t\ge t_0.
	$$
	Hence \(\|x(t)-x^\star\|\) is bounded on \([t_0,\infty)\), and so \(x(t)\) is bounded on \([t_0,\infty)\).\\
	
	\noindent The boundedness of \(\dot x(t)\) follows from the phase relation: $
	\dot x(t)
	=
	\frac1t
	\bigl(
	y(t)-\lambda(x(t)-z(t))
	\bigr).$
\end{proof}

\begin{theorem}[Point convergence]
	\label{thm:point-convergence-strict-gamma}
	Assume that \(F:\mathbb R^d\to\mathbb R\) is convex and continuously
	differentiable, that \(\nabla F\) is locally Lipschitz, and that
	$        \arg\min F\neq\emptyset.$ Let \((x,z)\) be the global strong solution of \eqref{eq:system7}. Assume further that $        t_0>0,~ \lambda=\alpha-1,~ \alpha\ge3, ~       \frac12<a<1,$ and $\gamma>0,~ \eta\ge0.$
	Then there exists \(x_\infty\in\arg\min F\) such that
	$$
	x(t)\to x_\infty
	\qquad\text{as }t\to\infty .
	$$
\end{theorem}

\begin{proof}
	Set
	$$
	S:=\arg\min F,\qquad
	F^\star:=\min_{\mathbb R^d}F,\qquad
	\mu:=2a-1 .
	$$
	Fix an arbitrary minimizer \(x^\star\in S\). By Theorem \ref{thm:value-estimate-a-section5}, we get        $ F(x(t))\to F^\star.$
	Since \(x(t)\) is bounded (by Lemma \ref{lem:boundedness-trajectories-recentered-phase}), every sequence \(t_n\to\infty\) has a subsequence
	along which \(x(t_n)\) converges.  If \(x(t_n)\to\bar x\), then by continuity of
	\(F\),
	$$
	F(\bar x)
	=
	\lim_{n\to\infty}F(x(t_n))
	=
	F^\star .
	$$
	Thus \(\bar x\in S\).  Hence, every cluster point of \(x(t)\) belongs to \(S\),
	and in particular
	$$
	\operatorname{dist}(x(t),S)\to0 .
	$$
	We next prove convergence of anchor distance differences.  For each \(p\in S\),
	define
	$$
	E_{a,p}(t)
	=
	\mu t^2\bigl(F(x(t))-F^\star\bigr)
	+
	\frac12\|y(t)\|^2
	+
	\frac{\lambda^2\mu}{2(1-a)}
	\|z(t)-p\|^2 .
	$$
	Then, by Proposition \ref{prop:accelerated-lyapunov-a-section5}, \(E_{a,p}(t)\) is nonincreasing and bounded below, and hence
	$$
	\lim_{t\to\infty}E_{a,p}(t)
	$$
	exists and is finite. Fix \(p,q\in S\).  Subtracting the two energies gives
	$$
	E_{a,p}(t)-E_{a,q}(t)
	=
	\frac{\lambda^2\mu}{2(1-a)}
	\left(
	\|z(t)-p\|^2-\|z(t)-q\|^2
	\right).
	$$
	Since both energies have finite
	limits, 
	$$
	\|z(t)-p\|^2-\|z(t)-q\|^2
	$$
	has a finite limit.  Since
	$$
	\|z(t)-p\|^2-\|z(t)-q\|^2
	=
	2\langle z(t),q-p\rangle
	+
	\|p\|^2-\|q\|^2,
	$$
	we obtain that
	$$
	\lim_{t\to\infty}\langle z(t),q-p\rangle
	$$
	exists for every \(p,q\in S\).\\
	
	\noindent {\it We now transfer this convergence from \(z(t)\) to \(x(t)\)}.  Fix \(p,q\in S\)
	and set  $d:=q-p.$
	Define
	$$
	\xi(t):=\langle x(t),d\rangle,
	\qquad
	\zeta(t):=\langle z(t),d\rangle,
	\qquad
	\chi(t):=\langle y(t),d\rangle .
	$$
	Thus, $  \zeta(t)\to\zeta_\infty$
	for some \(\zeta_\infty\in\mathbb R\). From the anchor equation: $
	\dot z(t)
	=
	-\frac{1-a}{\lambda}t\nabla F(x(t)),$
	we get
	$$
	\dot\zeta(t)
	=
	-\frac{1-a}{\lambda}t
	\langle\nabla F(x(t)),d\rangle .
	$$
	Thus
	$$
	t\langle\nabla F(x(t)),d\rangle
	=
	-\frac{\lambda}{1-a}\dot\zeta(t).
	$$
	Taking the inner product of the \(y\)-equation with \(d\) and using \eqref{eq:first-order-section4}, we get
	$\dot{y}(t)
	=
	(1-2a)t\nabla F(x(t))
	-\gamma y(t)
	-\frac{\eta}{t}D_M\bigl(x(t)-z(t)\bigr)y(t)$
	
	$$
	\begin{aligned}
		\dot\chi(t)
		&=
		-\mu t\langle\nabla F(x(t)),d\rangle
		-
		\gamma\chi(t)
		-
		\frac{\eta}{t}
		\left\langle
		D_M(x(t)-z(t))y(t),d
		\right\rangle .
	\end{aligned}
	$$

	Therefore
	$$
	\dot\chi(t)+\gamma\chi(t)
	=
	\beta\dot\zeta(t)
	-
	\rho(t),
	$$
	where
	$$
	\rho(t)
	:=
	\frac{\eta}{t}
	\left\langle
	D_M(x(t)-z(t))y(t),d
	\right\rangle ,~~~ \beta:=\frac{\lambda\mu}{1-a}.
	$$
	{\bf Claim 1:} \(\rho\in L^1(t_0,\infty)\). The proof of this claim is given in Appendix \ref{claim1}. \\
	
	\noindent For \(t\ge T\ge t_0\), variation of constants gives
	$$
	\chi(t)
	=
	e^{-\gamma(t-T)}\chi(T)
	+
	\beta\int_T^t e^{-\gamma(t-s)}\dot\zeta(s)\,ds
	-
	\int_T^t e^{-\gamma(t-s)}\rho(s)\,ds .
	$$
	The first term tends to zero because \(\gamma>0\).  The last term also tends to
	zero, since \(\rho\in L^1(t_0,\infty)\) and convolution with an exponentially
	decaying kernel vanishes at infinity. \noindent It remains to handle the middle term. For this, we employ integration by parts 
	$$
	\int_T^t e^{-\gamma(t-s)}\dot\zeta(s)\,ds
	=
	\zeta(t)
	-
	e^{-\gamma(t-T)}\zeta(T)
	-
	\gamma\int_T^t e^{-\gamma(t-s)}\zeta(s)\,ds .
	$$
	Since \(\zeta(t)\to\zeta_\infty\), we have
	$$
	e^{-\gamma(t-T)}\zeta(T)\to0~~\mbox{and}~~
	\gamma\int_T^t e^{-\gamma(t-s)}\zeta(s)\,ds
	\to
	\zeta_\infty .
	$$
	Therefore
	$$
	\int_T^t e^{-\gamma(t-s)}\dot\zeta(s)\,ds
	\to
	\zeta_\infty-0-\zeta_\infty
	=
	0 .
	$$
	Consequently,
	$$
	\chi(t)\to0 .
	$$
	That is,
	$$
	\langle y(t),q-p\rangle\to0
	\qquad\text{for every }p,q\in S.\\
	$$
	
	\noindent   Taking the inner product of $        \dot x(t)
	=
	\frac1t\bigl(y(t)-\lambda(x(t)-z(t))\bigr)$
	with \(d=q-p\), we obtain
	\begin{eqnarray}\label{xi}
		\dot\xi(t)+\frac{\lambda}{t}\xi(t)
		=
		\frac1t\bigl(\chi(t)+\lambda\zeta(t)\bigr).
	\end{eqnarray}
	Solving the linear equation \eqref{xi} for \(\xi(t)\), and for \(t\ge T\), gives
	$$
	\xi(t)
	=
	\left(\frac{T}{t}\right)^\lambda\xi(T)
	+
	\frac1{t^\lambda}
	\int_T^t
	s^{\lambda-1}
	\bigl(\chi(s)+\lambda\zeta(s)\bigr)\,ds .
	$$
	The first term tends to zero. Also, since $        \chi(t)\to0
	~~\text{and}~~
	\zeta(t)\to\zeta_\infty,$
	we have $        \chi(t)+\lambda\zeta(t)\to\lambda\zeta_\infty .
	$ Thus, by the weighted Cesaro principle, we obtain
	$$
	\frac1{t^\lambda}
	\int_T^t
	s^{\lambda-1}
	\bigl(\chi(s)+\lambda\zeta(s)\bigr)\,ds
	\to
	\zeta_\infty .
	$$
	Hence
	$$
	\xi(t)\to\zeta_\infty .
	$$
	Therefore, for every \(p,q\in S\),
	$$
	\lim_{t\to\infty}\langle x(t),q-p\rangle
	$$
	exists. Now, since
	\(x(t)\) is bounded, it has at least one cluster point.  Suppose that \(u\) and
	\(v\) are two cluster points of \(x(t)\). Then, from the first part of the proof, $  u,v\in S.$
	Choose sequences \(t_n\to\infty\) and \(s_n\to\infty\) such that $
	x(t_n)\to u,$ and $
	x(s_n)\to v .$
	Set        $d:=v-u.$
	Since \(u,v\in S\), the preceding conclusion applies with \(p=u\) and
	\(q=v\).  Hence $
	\langle x(t),v-u\rangle$
	has a finite limit as \(t\to\infty\).  
	Passing to the limit along \(t_n\) gives
	$$
	\lim_{n\to\infty}\langle x(t_n),v-u\rangle
	=
	\langle u,v-u\rangle .
	$$
	Passing to the limit along \(s_n\) gives
	$$
	\lim_{n\to\infty}\langle x(s_n),v-u\rangle
	=
	\langle v,v-u\rangle .
	$$
	Since the scalar function has a unique limit, these two values are equal:
	$$
	\langle u,v-u\rangle
	=
	\langle v,v-u\rangle .
	$$
	Therefore, \(u=v\). Thus, \(x(t)\) has exactly one cluster point,
	say \(x_\infty\), and so
	$$
	x(t)\to x_\infty \in S=\arg\min F .
	$$
\end{proof}

\begin{remark}
	The assumption $\gamma>0$ is used both to obtain $\rho\in L^1(t_0,\infty)$ (Appendix \ref{claim1}) and in the exponential-kernel argument proving $\chi(t)\to0$.
\end{remark}

\subsection{Anchor-Induced Selection on Affine Solution Sets}\label{Selection}

The preceding theorem proves that the primal trajectory converges strongly to a minimizer of $F$. We now show that, when the solution set is affine, the endogenous anchor does more than guarantee convergence: it selects a distinguished minimizer determined by the initial anchor $z_0$. This selection property is one of the main structural benefits of the anchor-memory mechanism. In contrast with the fixed-anchor acceleration \cite{SPR}, classical Nesterov-type dynamics \cite{SuBoydCandes16}, and Tikhonov-regularized inertial systems \cite{ACR,AttouchLaszlo24}, where the limiting minimizer is generally not explicitly controlled by an {\it internal anchor state}, the present dynamics (system \eqref{eq:system7}) preserves the affine component of the anchor and transfers this preserved information to the primal trajectory.\\

\noindent Throughout this subsection, we assume that $S$ is a nonempty affine subspace of $\mathbb R^d$. Thus there exist $p_0\in S$ and a linear subspace $L\subset\mathbb R^d$ such that
$$
S=p_0+L.
$$
We denote by $P_S$ the Euclidean projection onto $S$.

\begin{lemma}
	\label{lem:gradient-orthogonal-affine-solution-directions}
	Assume that $F:\mathbb R^d\to\mathbb R$ is convex and continuously differentiable, and that
	$
	S=\arg\min F=p_0+L
	$
	is a nonempty affine subspace. Then, for every $x\in\mathbb R^d$ and every $v\in L$,
	$$
	\langle \nabla F(x),v\rangle=0.
	$$
\end{lemma}

\begin{proof}
	Let $x\in\mathbb R^d$ and $v\in L$. Since $S=p_0+L$, we have
	$$
	p_0+sv\in S
	\qquad\text{for every }s\in\mathbb R.
	$$
	Because $F$ is convex and differentiable, the first-order convexity inequality gives
	$$
	F(p_0+sv)
	\ge
	F(x)+\langle \nabla F(x),p_0+sv-x\rangle .
	$$
	Since every point of $S$ is a minimizer, 
	$$
	F^\star
	\ge
	F(x)+\langle \nabla F(x),p_0-x\rangle
	+
	s\langle \nabla F(x),v\rangle .
	$$
	Equivalently,
	$$
	s\langle \nabla F(x),v\rangle
	\le
	F^\star-F(x)-\langle \nabla F(x),p_0-x\rangle .
	$$
	The right-hand side is independent of $s$, while $s\in\mathbb R$ is arbitrary. Therefore the coefficient of $s$ must vanish. Thus
	$$
	\langle \nabla F(x),v\rangle=0.
	$$
\end{proof}

\begin{theorem}[Anchor-induced projection selection]
	\label{thm:anchor-induced-projection-selection}
	Assume that $F:\mathbb R^d\to\mathbb R$ is convex and continuously differentiable, that $\nabla F$ is locally Lipschitz, and that
	$
	S:=\arg\min F
	$
	is a nonempty affine subspace of $\mathbb R^d$. Let $(x,y,z)$ be the global strong solution of \eqref{eq:system7}
	with parameters
	$ t_0>0,~
	\alpha\ge 3, ~
	\lambda=\alpha-1,~
	\frac12<a<1,~
	\gamma>0,~
	\eta\ge0.
	$
	Then the primal trajectory converges to the projection of the initial anchor onto the solution set:
	$$
	x(t)\longrightarrow P_S(z_0)
	\qquad\text{as }t\to\infty.
	$$
\end{theorem}

\begin{proof}
	By Theorem~\ref{thm:point-convergence-strict-gamma}, there exists some point $x_\infty\in S$ such that
	\begin{eqnarray}\label{PC}
		x(t)\to x_\infty  \qquad \text{as }t\to\infty.
	\end{eqnarray}
	It remains to identify $x_\infty$.\\
	
	\noindent Since $S$ is affine, let $v\in L$, then by Lemma~\ref{lem:gradient-orthogonal-affine-solution-directions},
	$$
	\langle \nabla F(x(t)),v\rangle=0
	\qquad\text{for every }t\ge t_0.
	$$
	Taking the inner product of the anchor equation with $v$, we obtain
	$$
	\frac{d}{dt}\langle z(t),v\rangle
	=
	-\frac{t}{\lambda}(1-a)
	\langle \nabla F(x(t)),v\rangle
	=
	0.
	$$
	Hence, the affine component of the anchor is conserved:
	\begin{eqnarray}\label{z0}
		\langle z(t),v\rangle
		=
		\langle z_0,v\rangle
		\qquad
		\text{for every }t\ge t_0
		\quad\text{and every }v\in L.    
	\end{eqnarray}
	We now transfer this conserved anchor information to the primal trajectory. \\
	
	\noindent In the proof of Theorem~\ref{thm:point-convergence-strict-gamma}, it was shown that for every $p,q\in S$,
	\begin{eqnarray}\label{z01}
		\lim_{t\to\infty}\langle x(t),q-p\rangle
		=
		\lim_{t\to\infty}\langle z(t),q-p\rangle.  
	\end{eqnarray}
	Since every vector $v\in L$ can be written as $v=q-p$ for some $p,q\in S$, it follows from \eqref{PC}, \eqref{z0} and \eqref{z01} that 
	$$
	\langle x_\infty,v\rangle
	=
	\langle z_0,v\rangle
	\qquad\text{for every }v\in L.
	$$
	That is,
	$$
	\langle x_\infty-z_0,v\rangle=0
	\qquad\text{for every }v\in L.
	$$
	Thus
	$$
	z_0-x_\infty\in L^\perp.
	$$
	Since $x_\infty\in S=p_0+L$, this is precisely the variational characterization of the Euclidean projection of $z_0$ onto the affine subspace $S$. Therefore
	$$
	x_\infty=P_S(z_0).
	$$
	Consequently,
	$$
	x(t)\to P_S(z_0)
	\qquad\text{as }t\to\infty.
	$$
\end{proof}

\begin{corollary}[Rank-deficient least-squares problems]
	\label{cor:rank-deficient-least-squares-selection}
	Let
	$
	F(x)=\frac12\|Ax-b\|^2,
	$
	where $A\in\mathbb R^{m\times d}$ and $b\in\mathbb R^m$. Then the least-squares solution set
	$
	S:=\arg\min F
	$
	is a nonempty affine subspace of $\mathbb R^d$. More precisely, if $x_\star\in S$, then
	$
	S=x_\star+\ker A.
	$
	Under the assumptions of Theorem~\ref{thm:anchor-induced-projection-selection}, the solution of system \eqref{eq:system7} satisfies
	$$
	x(t)\longrightarrow P_S(z_0)
	\qquad\text{as }t\to\infty.
	$$
	In particular, if $A$ is rank deficient, then the least-squares minimizer is not unique, and the dynamics \eqref{eq:system7} selects, among all least-squares minimizers, the unique one that is closest in Euclidean norm to the initial anchor $z_0$.
\end{corollary}

\subsection{Quantitative Restoring--Damping Effects}
\label{subsec:quantitative-restoring-damping}

Set $r(t):=x(t)-z(t).$ We now make precise the stabilizing role of the nonlinear restoring-damping coefficient \(\eta\).  The following estimate shows that \(\eta\) imposes a finite logarithmic budget on large coordinatewise phase amplitudes in regions where the anchor displacement is non-negligible.

\begin{proposition}\label{PA1}
	Assume that \(F:\mathbb R^d\to\mathbb R\) is convex and continuously differentiable, that \(\nabla F\) is locally Lipschitz, and that
	$
	\arg\min F\neq\emptyset.$
	Fix $x^\star\in\arg\min F, ~~
	F^\star:=F(x^\star)=\min_{\mathbb R^d}F.$
	Let $
	t_0>0, ~
	\alpha\ge3, ~
	\lambda=\alpha-1, ~
	\frac12<a<1,~
	\gamma\ge0, ~
	\eta\ge0.$ Let \((x,y,z)\) be the global strong solution of \eqref{eq:system7}, and let \(E_{a,x^\star}\) denote the Lyapunov functional in \eqref{eq:Ea-section5}. Then, for every \(T\ge t_0\),
	$$
	\gamma\int_{t_0}^{T}\|y(t)\|^2\,dt
	+
	\eta
	\sum_{i=1}^d m_i^2
	\int_{t_0}^{T}
	\frac{r_i(t)^2y_i(t)^2}{t}\,dt
	\le
	E_{a,x^\star}(t_0)-E_{a,x^\star}(T).
	$$
	
	Consequently,
	$$
	\eta
	\sum_{i=1}^d m_i^2
	\int_{t_0}^{\infty}
	\frac{r_i(t)^2y_i(t)^2}{t}\,dt
	\le
	E_{a,x^\star}(t_0).
	$$
	If \(\eta>0\), \(m_i\neq0\), and \(\delta>0\), then
	$$
	\int_{\{t\ge t_0:\ |r_i(t)|\ge\delta\}}
	\frac{y_i(t)^2}{t}\,dt
	\le
	\frac{E_{a,x^\star}(t_0)}{\eta m_i^2\delta^2}.
	$$
	More generally, for every \(\varepsilon>0\),
	$$
	\int_{\{t\ge t_0:\ |r_i(t)|\ge\delta,\ |y_i(t)|\ge\varepsilon\}}
	\frac{dt}{t}
	\le
	\frac{E_{a,x^\star}(t_0)}
	{\eta m_i^2\delta^2\varepsilon^2}.
	$$
	Thus, for every fixed \(\delta>0\) and \(\varepsilon>0\), large phase amplitudes in the \(i\)-th coordinate cannot persist on a set of infinite logarithmic measure while the trajectory remains separated from the anchor in that coordinate.
\end{proposition}

\begin{proof}
	From \eqref{eq:Ea-drop-section5}, we obtain
	\begin{eqnarray}\label{P1}
		\dot E_{a,x^\star}(t)
		\le
		-\gamma\|y(t)\|^2
		-\frac{\eta}{t}
		\langle D_M(r(t))y(t),y(t)\rangle .    
	\end{eqnarray}
	Using the diagonal form of \(D_M\) (see \eqref{eq:DM-definition} and \eqref{eq:DM-dissipation}), \eqref{P1} becomes
	\begin{eqnarray}\label{P2}
		\dot E_{a,x^\star}(t)
		\le
		-\gamma\|y(t)\|^2
		-
		\eta
		\sum_{i=1}^d m_i^2
		\frac{r_i(t)^2y_i(t)^2}{t}.
	\end{eqnarray}
	Integrating \eqref{P2} from \(t_0\) to \(T\) gives
	$$
	E_{a,x^\star}(T)-E_{a,x^\star}(t_0)
	\le
	-\gamma\int_{t_0}^{T}\|y(t)\|^2\,dt
	-
	\eta
	\sum_{i=1}^d m_i^2
	\int_{t_0}^{T}
	\frac{r_i(t)^2y_i(t)^2}{t}\,dt .
	$$
	Rearranging yields
	\begin{eqnarray}\label{P3}
		\gamma\int_{t_0}^{T}\|y(t)\|^2\,dt
		+
		\eta
		\sum_{i=1}^d m_i^2
		\int_{t_0}^{T}
		\frac{r_i(t)^2y_i(t)^2}{t}\,dt
		\le
		E_{a,x^\star}(t_0)-E_{a,x^\star}(T).
	\end{eqnarray}
	Dropping the nonnegative \(\gamma\)-term  and using $E_{a,x^\star}(T)\ge0$ 
	for every \(T\ge t_0\), \eqref{P3} becomes
	$$
	\eta
	\sum_{i=1}^d m_i^2
	\int_{t_0}^{T}
	\frac{r_i(t)^2y_i(t)^2}{t}\,dt
	\le
	E_{a,x^\star}(t_0).
	$$
	Letting \(T\to\infty\), and using monotone convergence for the nonnegative integrands, gives
	\begin{eqnarray}\label{P5}
		\eta
		\sum_{i=1}^d m_i^2
		\int_{t_0}^{\infty}
		\frac{r_i(t)^2y_i(t)^2}{t}\,dt
		\le
		E_{a,x^\star}(t_0).    
	\end{eqnarray}
	Now fix an index \(i\in\{1,\ldots,d\}\), and assume that \(\eta>0\), \(m_i\neq0\), and \(\delta>0\). 
	Define
	$$
	A_{i,\delta}
	:=
	\{t\ge t_0:\ |r_i(t)|\ge\delta\}.
	$$
	On \(A_{i,\delta}\), we have $r_i(t)^2\ge\delta^2.$
	Hence, from \eqref{P5}, we get
	$$
	\eta m_i^2\delta^2
	\int_{A_{i,\delta}}
	\frac{y_i(t)^2}{t}\,dt
	\le
	\eta m_i^2
	\int_{A_{i,\delta}}
	\frac{r_i(t)^2y_i(t)^2}{t}\,dt
	\le
	E_{a,x^\star}(t_0).
	$$
	Therefore,
	\begin{eqnarray}\label{P6}
		\int_{A_{i,\delta}}
		\frac{y_i(t)^2}{t}\,dt
		\le
		\frac{E_{a,x^\star}(t_0)}
		{\eta m_i^2\delta^2}.    
	\end{eqnarray}
	Finally, let \(\varepsilon>0\), and define
	$$
	B_{i,\delta,\varepsilon}
	:=
	\{t\ge t_0:\ |r_i(t)|\ge\delta,\ |y_i(t)|\ge\varepsilon\}.
	$$
	On \(B_{i,\delta,\varepsilon}\), we also have $
	y_i(t)^2\ge\varepsilon^2.$
	Thus, from \eqref{P6} we get
	$$
	\varepsilon^2
	\int_{B_{i,\delta,\varepsilon}}
	\frac{dt}{t}
	\le
	\int_{B_{i,\delta,\varepsilon}}
	\frac{y_i(t)^2}{t}\,dt
	\le
	\frac{E_{a,x^\star}(t_0)}
	{\eta m_i^2\delta^2}.
	$$
	This completes the proof. 
\end{proof}

\begin{proposition}[Coordinatewise damping amplification]\label{prop:coordinatewise-damping-amplification}
	Under the assumptions of Proposition \ref{PA1}, the \(i\)-th component of the phase variable satisfies for every \(t\ge s\ge t_0\),
	$$
	\begin{aligned}
		y_i(t)
		&=
		\exp\left(
		-\int_s^t
		\left[
		\gamma+\frac{\eta m_i^2 r_i(\tau)^2}{\tau}
		\right]d\tau
		\right)y_i(s)
		\\
		&\quad
		-(2a-1)
		\int_s^t
		\exp\left(
		-\int_\tau^t
		\left[
		\gamma+\frac{\eta m_i^2 r_i(\sigma)^2}{\sigma}
		\right]d\sigma
		\right)
		\tau\,\partial_iF(x(\tau))\,d\tau .
	\end{aligned}
	$$
	If, for some \(\delta>0\), $
	|r_i(\tau)|\ge\delta
	~\text{for every }\tau\in[s,t],$
	then the damping factor satisfies 
	$$
	\exp\left(
	-\int_s^t
	\left[
	\gamma+\frac{\eta m_i^2 r_i(\tau)^2}{\tau}
	\right]d\tau
	\right)
	\le
	e^{-\gamma(t-s)}
	\left(\frac{s}{t}\right)^{\eta m_i^2\delta^2}.
	$$
	In particular, if \(\eta>0\) and \(m_i\neq0\), then whenever the trajectory remains separated from the anchor in the \(i\)-th coordinate, the nonlinear restoring-damping term produces an additional polynomial damping factor beyond the linear damping factor \(e^{-\gamma(t-s)}\).
\end{proposition}

\begin{proof}
	From the phase equation \eqref{eq:first-order-section4} and \eqref{eq:DM-definition},
	we obtain, componentwise
	$$
	\dot y_i(t)
	=
	(1-2a)t\,\partial_iF(x(t))
	-\gamma y_i(t)
	-\frac{\eta m_i^2r_i(t)^2}{t}y_i(t).
	$$
	That is,
	$$
	\dot y_i(t)
	+
	\left(
	\gamma+\frac{\eta m_i^2r_i(t)^2}{t}
	\right)y_i(t)
	=
	-(2a-1)t\,\partial_iF(x(t)).
	$$
	This is a scalar nonautonomous linear differential equation. Thus, applying the variation-of-constants formula gives, for every \(t\ge s\ge t_0\),
	\begin{eqnarray}  \label{H1}
		\begin{aligned}
			y_i(t)
			&=
			\exp\left(
			-\int_s^t
			\left[
			\gamma+\frac{\eta m_i^2 r_i(\tau)^2}{\tau}
			\right]d\tau
			\right)y_i(s)
			\\
			&\quad
			-(2a-1)
			\int_s^t
			\exp\left(
			-\int_\tau^t
			\left[
			\gamma+\frac{\eta m_i^2 r_i(\sigma)^2}{\sigma}
			\right]d\sigma
			\right)
			\tau\,\partial_iF(x(\tau))\,d\tau .
		\end{aligned}
	\end{eqnarray}
	The homogeneous part of \eqref{H1} therefore contains the damping multiplier
	\begin{eqnarray}\label{H2}
		\exp\left(
		-\int_s^t
		\left[
		\gamma+\frac{\eta m_i^2 r_i(\tau)^2}{\tau}
		\right]d\tau
		\right).    
	\end{eqnarray}
	Now suppose that $
	|r_i(\tau)|\ge\delta
	~\text{for every }\tau\in[s,t].$
	Then $
	r_i(\tau)^2\ge\delta^2
	~\text{for every }\tau\in[s,t].$ Consequently,
	$$
	\int_s^t
	\left[
	\gamma+\frac{\eta m_i^2r_i(\tau)^2}{\tau}
	\right]d\tau
	\ge
	\gamma(t-s)
	+
	\eta m_i^2\delta^2
	\int_s^t\frac{d\tau}{\tau}.
	$$
	Thus, we obtain a bound for \eqref{H2} as follows:
	$$
	\exp\left(
	-\int_s^t
	\left[
	\gamma+\frac{\eta m_i^2r_i(\tau)^2}{\tau}
	\right]d\tau
	\right)
	\le
	e^{-\gamma(t-s)}
	\exp\left(
	-\eta m_i^2\delta^2
	\log\left(\frac{t}{s}\right)
	\right).
	$$
	Therefore,
	$$
	\exp\left(
	-\int_s^t
	\left[
	\gamma+\frac{\eta m_i^2r_i(\tau)^2}{\tau}
	\right]d\tau
	\right)
	\le
	e^{-\gamma(t-s)}
	\left(\frac{s}{t}\right)^{\eta m_i^2\delta^2}.
	$$
	If \(\eta>0\) and \(m_i\neq0\), then the exponent \(\eta m_i^2\delta^2\) is strictly positive, and the factor
	$
	\left(\frac{s}{t}\right)^{\eta m_i^2\delta^2}
	$
	is an additional polynomial damping factor. This proves the result.
\end{proof}

\begin{remark}[Interpretation of the nonlinear restoring-damping mechanism]
	\label{rem:nonlinear-damping-interpretation}
	Propositions~\ref{PA1} and~\ref{prop:coordinatewise-damping-amplification} quantify the stabilizing effect of the nonlinear parameter $\eta$. Since the $i$-th phase equation contains the effective damping coefficient
	$$
	\gamma+\frac{\eta m_i^2|x_i(t)-z_i(t)|^2}{t},
	$$
	the nonlinear damping is stronger in coordinates where the trajectory is farther from the endogenous anchor.\\
	
	\noindent Proposition~\ref{PA1} shows that, when $\eta>0$, large phase amplitudes cannot persist for infinite logarithmic time in coordinates that remain significantly displaced from the anchor. Proposition~\ref{prop:coordinatewise-damping-amplification} complements this result by showing that, on intervals where $|x_i(t)-z_i(t)|$ stays bounded away from zero, the nonlinear term contributes an additional polynomial attenuation to the homogeneous phase response. Thus, these propositions provide a quantitative description of the additional attenuation induced by the restoring--damping mechanism away from the anchor, while preserving the accelerated objective-value decay
	$
	F(x(t))-F^\star=\mathcal O(t^{-2}).
	$
\end{remark}

%\subsection{Stronger geometric assumptions on $F$: distance estimates, gradient rates, and the selection program}

%\section{Parameter regimes and dynamical classification}

\section{Numerical Experiments}
\label{sec:numerical-experiments}

We illustrate three main features of the proposed dynamics: anchor-induced selection, preservation of accelerated objective-value decay, and nonlinear
coordinatewise restoring-damping. All experiments for our proposed dynamics are based on \eqref{eq:first-order-section4}; the integration interval and solver settings are specified in each experiment.  \\

\noindent The numerical comparisons   focus on Nesterov's ODE and Hessian-driven damping, as these provide the most direct reference models for the selection and oscillation-suppression effects studied here. The other models discussed in Section \ref{sec:relation-existing-models} involve structurally different selection or feedback mechanisms and are therefore not used as direct numerical
benchmarks.

\subsection{Anchor-induced selection}
\label{subsec:anchor-induced-selection}

Consider the underdetermined least-squares problem
$
F(x)=\frac12(x_1+x_2-1)^2,
$
corresponding to
$
A=
\begin{pmatrix}
	1&1
\end{pmatrix},
~
b=1,
~
\operatorname{rank}A=1<d=2
$ (see Corollary \ref{cor:rank-deficient-least-squares-selection}). $F$ is convex and continuously differentiable, with globally
Lipschitz gradient, and its minimizer set $S$ is nonempty and affine:
$
S=\argmin F
=
\left\{
x\in\mathbb R^2:x_1+x_2=1
\right\}.
$
Thus the problem has infinitely many minimizers and is well suited for
illustrating the minimizer-selection mechanism of the proposed dynamics.
The Euclidean projection onto $S$ is
$$
P_S(z)
=
z-\frac{z_1+z_2-1}{2}
\begin{pmatrix}
	1\\
	1
\end{pmatrix}.
$$
We fix
$
t_0=1,~
\alpha=3,~
\lambda=2,~
a=\frac34,
$
and use the same initial primal position and velocity in all runs,
$
x_0=
\begin{pmatrix}
	-2\\
	2
\end{pmatrix},
~~
v_0=0,
$
with
$
\gamma=1,
~~
\eta=1,~~
M=I_2.
$
We then vary the initial anchor, choosing
$$
z_0^{(1)}
=
\begin{pmatrix}
	-1\\
	1
\end{pmatrix},
\qquad
z_0^{(2)}
=
\begin{pmatrix}
	0\\
	0
\end{pmatrix},
\qquad
z_0^{(3)}
=
\begin{pmatrix}
	1\\
	-1
\end{pmatrix}.
$$
The trajectories are integrated on $[1,100]$ using the adaptive DOP853 Runge--Kutta method with relative tolerance $10^{-10}$ and absolute tolerance $10^{-12}$.\\
All three anchors satisfy
$
z_{0,1}^{(j)}+z_{0,2}^{(j)}=0,
$
and therefore lie on the affine line
$
\{z\in\mathbb R^2:z_1+z_2=0\},
$
which is parallel to $S$. Hence the three anchors have the same signed
normal distance from $S$ while differing in their tangential components.
 The hypotheses relevant to the affine minimizer-selection result are
satisfied. By
Theorem~\ref{thm:anchor-induced-projection-selection}, the proposed
dynamics ensures $
x(t)\longrightarrow P_S(z_0).
$
Consequently, although the objective and the independent initial primal data $x_0,v_0$ are identical in all three runs, changing the initial anchor changes the minimizer selected asymptotically, as seen in Figure \ref{fig:experiment-selection}(a). 

 For reference, we also integrate the classical Nesterov flow \eqref{NODE}, and obtain 
$
x(t)\longrightarrow
\begin{pmatrix}
	-\frac32\\[1mm]
	\frac52
\end{pmatrix}
=
P_S(x_0).
$
The Nesterov trajectory is included only as a structural reference; the
comparison is not intended as a performance ranking between the two
dynamics.

 Table~\ref{tab:experiment-A-summary} reports the numerical states at
$t=100$ together with their distances from the corresponding reference
minimizers. For the proposed dynamics, all three terminal states are within
$7\times10^{-4}$ of the theoretically predicted projections
$P_S(z_0^{(j)})$. In particular, the three runs approach distinct points of
the same minimizer set despite having identical $x_0$ and $v_0$. This
provides a direct numerical illustration of the anchor-induced projection
selection established in
Theorem~\ref{thm:anchor-induced-projection-selection}.

\begin{table}[htbp]
	\centering
	\small
	\caption{Anchor-induced selection at $t=100$.}
	\label{tab:experiment-A-summary}
	\begin{tabular}{@{}llll@{}}
		\toprule
		Case & Reference minimizer & $x(100)$ & Distance to reference \\
		\midrule
		
		$z_0^{(1)}=(-1,1)$
		&
		$P_S(z_0^{(1)})=(-1/2,3/2)$
		&
		$(-0.5003584,1.5003584)$
		&
		$5.068\times10^{-4}$
		\\
		
		$z_0^{(2)}=(0,0)$
		&
		$P_S(z_0^{(2)})=(1/2,1/2)$
		&
		$(0.4995782,0.5004218)$
		&
		$5.966\times10^{-4}$
		\\
		
		$z_0^{(3)}=(1,-1)$
		&
		$P_S(z_0^{(3)})=(3/2,-1/2)$
		&
		$(1.4995677,-0.4995677)$
		&
		$6.113\times10^{-4}$
		\\
		
		\midrule
		
		Nesterov reference
		&
		$P_S(x_0)=(-3/2,5/2)$
		&
		$(-1.5006197,2.4993803)$
		&
		$8.764\times10^{-4}$
		\\
		
		\bottomrule
	\end{tabular}
\end{table}

\noindent Figure~\ref{fig:experiment-selection} provides the corresponding geometric
and quantitative views. Panel~(a) shows the trajectories in the
$(x_1,x_2)$-plane together with the affine minimizer set $S$. The three
proposed trajectories start from the same $x_0$ but approach the three
distinct projection points $P_S(z_0^{(j)})$. The Nesterov trajectory,
shown as a structural reference, follows a single anchor-independent path
and approaches $P_S(x_0)$ in this example. Thus panel~(a) makes visible the
additional selection flexibility introduced by the anchor variable.\\

\noindent Panel~(b) shows the corresponding projection-selection errors
$
\|x(t)-P_S(z_0)\|
$
on a logarithmic vertical scale. For all three anchor choices, the
numerically computed errors decrease smoothly and monotonically over the
interval $1\leq t\leq100$. This observed decay is consistent with
the projection-selection theorem and complements the geometric picture in
panel~(a) by directly measuring convergence toward the minimizer selected by
each initial anchor. 

\begin{figure}[htbp]
	\centering
	\includegraphics[width=\textwidth]
	{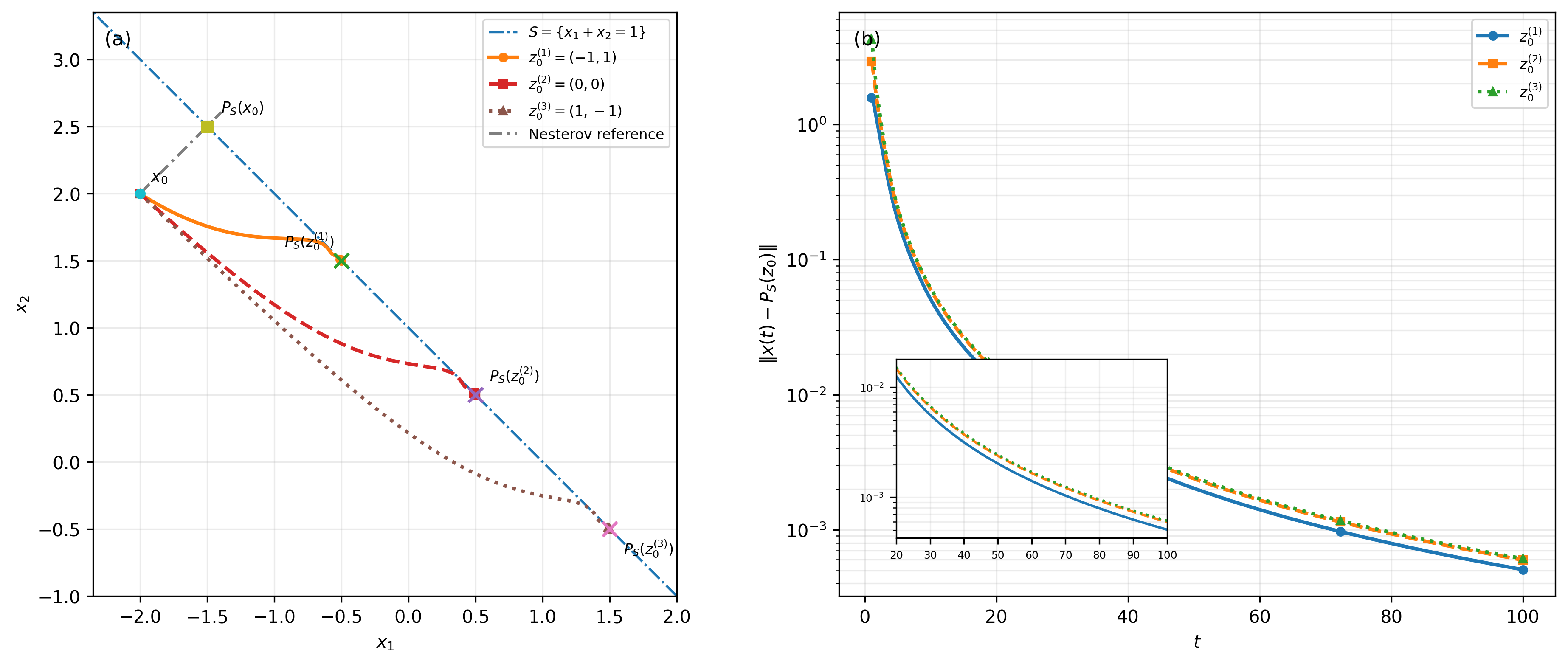}
	\caption{Anchor-induced minimizer selection.
		Panel~(a) shows the trajectories for three different initial anchors. Panel~(b) shows the corresponding selection errors $\|x(t)-P_S(z_0)\|$.}
	\label{fig:experiment-selection}
\end{figure}

\subsection{Accelerated decay under nonlinear restoring-damping}
\label{subsec:accelerated-decay}

Consider the scalar convex objective
$
F(x)=\frac14x^4,
~
F^\star=0,
~
x^\star=0.
$
This example provides a simple setting in which to examine the effect of the
nonlinear restoring-damping parameter $\eta$ on the transient dynamics while
monitoring the accelerated objective-value estimate.
\noindent Using
$
t_0=1,
~
\alpha=3,
~
\lambda=2,
~
a=\frac34,
$
we take
$
x_0=2,
~
z_0=\frac34,
~
v_0=0,
~
\gamma=\frac{3}{4},
~
M=1,
$
and compare
$$
\eta=0,
\qquad
\eta=\frac12,
\qquad
\eta=2.
$$
Theorem~\ref{thm:value-estimate-a-section5} gives 
$
F(x(t))-F^\star
\leq
\frac{C}{t^2},
$
where
$
C
=
F(x_0)
+
\frac{|y_0|^2}{2(2a-1)}
+
\frac{\lambda^2}{2(1-a)}
|z_0-x^\star|^2
=
\frac{59}{4}.
$\\

\noindent Numerically, the objective residuals are smooth and monotonically decreasing
over the displayed interval $1\leq t\leq100$ for all three parameter choices of $\eta$. 
The absence of a visible objective-value rebound in this experiment provides
a clean setting for comparing the different transient profiles generated by
the nonlinear restoring-damping term.\\

\noindent Table~\ref{tab:experiment-B-summary} reports the terminal objective values,
the compensated residuals at $t=100$, and the largest compensated values
observed over the simulated interval. In every case,
$
t^2\bigl(F(x(t))-F^\star\bigr)
$
remains well below the theoretical constant
$
C=\frac{59}{4}=14.75.
$ Also,  changing $\eta$ modifies the transient dynamics without destroying the common accelerated objective-value estimate.
\begin{table}[htbp]
	\centering
	\small
	\caption{Objective-value decay for different values of $\eta$. }
	\label{tab:experiment-B-summary}
	\begin{tabular}{@{}lllll@{}}
		\toprule
		$\eta$
		&
		$F(x(100))$
		&
		$100^2F(x(100))$
		&
		$\displaystyle\max_{1\leq t\leq100}t^2F(x(t))$
		&
		$\displaystyle\frac{\max t^2F(x(t))}{C}$
		\\
		\midrule
		
		$0$
		&
		$1.5991\times10^{-7}$
		&
		$1.5991\times10^{-3}$
		&
		$4.5584$
		&
		$3.090\times10^{-1}$
		\\
		
		$\frac12$
		&
		$1.0654\times10^{-7}$
		&
		$1.0654\times10^{-3}$
		&
		$4.4518$
		&
		$3.018\times10^{-1}$
		\\
		
		$2$
		&
		$4.8292\times10^{-8}$
		&
		$4.8292\times10^{-4}$
		&
		$4.2880$
		&
		$2.907\times10^{-1}$
		\\
		
		\bottomrule
	\end{tabular}
\end{table}

\noindent Figure~\ref{fig:experiment-rate} gives the corresponding graphical illustration. Panel~(a) shows
$
F(x(t))-F^\star
$
on logarithmic axes together with the common theoretical envelope
$
\frac{C}{t^2}.
$
All three residuals decrease smoothly over the simulated interval and remain below the same theoretical envelope.
Panel~(b) displays the compensated residual
$
t^2\bigl(F(x(t))-F^\star\bigr).
$
The horizontal line $C=59/4$ represents the theoretical Lyapunov bound.
The fact that all three compensated trajectories remain well below this
level illustrates numerically that the additional nonlinear
restoring-damping feedback is compatible with the accelerated $\mathcal O(t^{-2})$ objective-value estimate.

\begin{figure}[htbp]
	\centering
	\includegraphics[width=\textwidth]
	{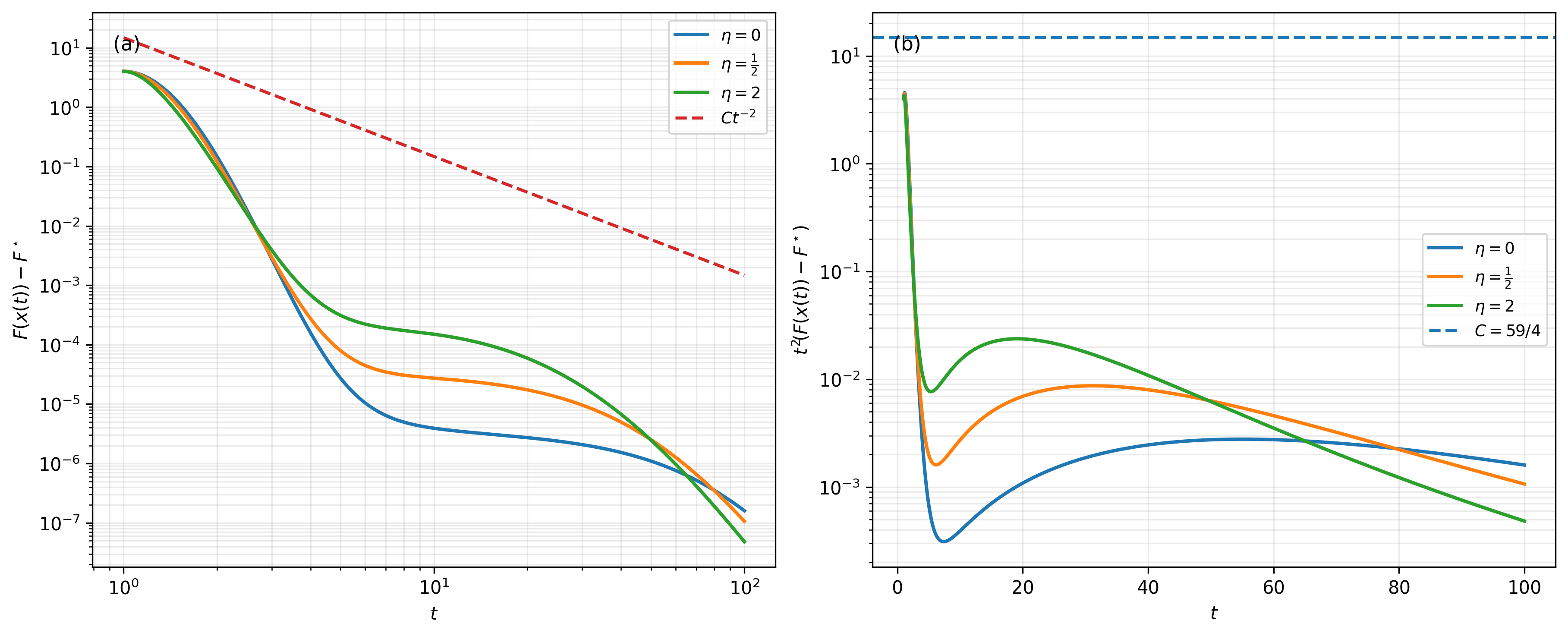}
	\caption{Accelerated objective-value decay under nonlinear
		restoring-damping. Panel~(a) shows $F(x(t))-F^\star$ on log-log axes, with the theoretical
		envelope $Ct^{-2}$. Panel~(b) shows the corresponding
		compensated residuals $t^2(F(x(t))-F^\star)$ with the theoretical bound $C$.}
	\label{fig:experiment-rate}
\end{figure}

\noindent This experiment therefore illustrates two complementary features of the proposed mechanism. First, introducing the nonlinear restoring-damping term does not alter the common $\mathcal O(t^{-2})$ objective-value guarantee. Second, the parameter $\eta$ provides a means of modifying the transient behavior while maintaining smooth objective decay; in the present example, increasing $\eta$ reduces the observed oscillations. Thus, the experiment demonstrates that the restoring-damping feedback can substantially modify the transient dynamics while remaining compatible with the accelerated rate established by the theory (Theorem~\ref{thm:value-estimate-a-section5}).

\subsection{Coordinatewise restoring-damping and oscillation suppression}
\label{subsec:coordinatewise-numerics}

We next examine the transient behavior induced by the coordinatewise anchor-displacement restoring-damping mechanism and compare it with the Nesterov ODE \eqref{NODE} and the Hessian-driven damping system \eqref{Hessian}. 

 We consider the quadratic function
$
F(x)
=
\frac12\left(x_1^2+10x_2^2\right),
~~
x=(x_1,x_2)^\top\in\mathbb{R}^2.
$
The unique minimizer and minimum value are
$
x^\star=(0,0)^\top,
~ F^\star=0.
$ \noindent For all three dynamics we take
$
t_0=1,
~
\alpha=3,
~
\lambda=\alpha-1=2,
$
with common primal initial data
$
x(1)
=
\begin{pmatrix}
	1\\
	1
\end{pmatrix},
~
\dot x(1)
=
\begin{pmatrix}
	0\\
	0
\end{pmatrix}.
$
 For the proposed dynamics, we use
$
a=0.95,
~
\gamma=0,
~\eta=60,
~
M=I_2,
$
and initialize the endogenous anchor at
$
z(1)
=
\begin{pmatrix}
	0.5\\
	0.8
\end{pmatrix}.
$ For comparison, we take  $\beta=0.15$ for the Hessian-driven damping system \eqref{Hessian}.

 The trajectories are integrated on $[1,8]$ using the high-order adaptive DOP853 Runge--Kutta method, with relative and absolute tolerances of $10^{-10}$ and $10^{-12}$, respectively, or tighter for
the tail-dissipation diagnostic. The main parameters and numerical
diagnostics are summarized in
Table~\ref{tab:coordinatewise-comparison}. A zero crossing is counted
whenever a coordinate changes sign on $(1,8]$.
\begin{table}[t]
	\centering
	\caption{Parameters and numerical diagnostics for the coordinatewise
		comparison on $[1,8]$.}
	\label{tab:coordinatewise-comparison}
	\renewcommand{\arraystretch}{1.18}
	\begin{tabular}{lccc}
		\hline
		& Proposed dynamics
		& Hessian-driven
		& Nesterov ODE
		\\
		\hline
		$a$ & $0.95$ & --- & --- \\
		$\gamma$ & $0$ & --- & --- \\
		$\eta$ & $60$ & --- & --- \\
		$\beta$ & --- & $0.15$ & $0$ \\
		$x_1$ zero crossings & $0$ & $2$ & $2$ \\
		$x_2$ zero crossings & $0$ & $6$ & $7$ \\
		$\displaystyle \min_{1\leq t\leq8}x_1(t)$
		& $0.17587$
		& $-0.12548$
		& $-0.16840$
		\\
		$\displaystyle \min_{1\leq t\leq8}x_2(t)$
		& $0.01448$
		& $-0.13673$
		& $-0.35687$
		\\
		$F(x(8))-F^\star$
		& $1.6513\times10^{-2}$
		& $5.7834\times10^{-4}$
		& $1.2484\times10^{-2}$
		\\
		\hline
	\end{tabular}
\end{table}

\noindent Figure~\ref{fig:expC-x1} compares the coordinate trajectories. For the proposed dynamics, both $x_1(t)$ and $x_2(t)$ remain positive and decrease smoothly throughout the displayed transient; $(x_1(t),  x_2(t))^\top$ does not cross the minimizer, $
x^\star=(0,0)^\top$. By contrast, both comparison systems overshoot the minimizer in both coordinates. In the first coordinate $x_1(t)$, the Hessian-driven and Nesterov trajectories each cross zero twice.  The contrast is even more pronounced in the second coordinate $x_2(t)$. The Hessian-driven trajectory crosses zero six times and reaches $\min_{1\leq t\leq8}x_2(t)\approx-0.13673,$ whereas Nesterov's ODE crosses zero seven times and reaches the
substantially larger excursion $ \min_{1\leq t\leq8}x_2(t) \approx-0.35687.$ Hence,  the Hessian-driven damping moderates the oscillatory behavior of Nesterov's ODE but does not eliminate it. However, the proposed dynamics removes the
overshoots and produces a nonoscillatory transient.

 Figure~\ref{fig:expC-x1z1} displays the motion of the endogenous anchor. Since
$
\dot z(t)
=
-\frac{t}{\lambda}(1-a)\nabla F(x(t)),
$
we obtain
$
\dot z_1(t)
=
-0.025\,t\,x_1(t)$ and $
\dot z_2(t)
=
-0.25\,t\,x_2(t).
$
Because $x_1(t)$ and $x_2(t)$ remain positive on
$[1,8]$, both anchor components decrease monotonically over the same
interval. These plots
therefore illustrate that the proposed dynamics combines a smooth primal trajectory with an evolving endogenous reference
state, rather than relying on a fixed anchor.
 Figure~\ref{fig:expC-objective} ({\bf Left}) 
shows the objective residual
$
F(x(t))-F^\star.
$
For the proposed dynamics, the numerically computed residual decreases
monotonically throughout $[1,8]$. By contrast, the Hessian-driven and
Nesterov trajectories display nonmonotone objective transients
associated with their repeated coordinate overshoots.   Figure~\ref{fig:expC-objective} ({\bf Right}) illustrates the remaining nonlinear damping activity in the second coordinate of the {\it proposed dynamics}. We define the coordinatewise tail dissipation as
$$
J_{2,\mathrm{tail}}(T)
=
\eta
\int_T^\infty
\frac{
	\bigl(x_2(t)-z_2(t)\bigr)^2y_2(t)^2
}{t}\,dt.
$$
\noindent Numerically, the infinite endpoint is approximated by
$T_{\max}=200$. Increasing the cutoff beyond $T_{\max}=200$
changes the reported values by less than the displayed precision. The
resulting values are
$
J_{2,\mathrm{tail}}(1)
\approx6.57048,$ and $
J_{2,\mathrm{tail}}(8)
\approx0.02840.
$
Thus,
$
\frac{J_{2,\mathrm{tail}}(8)}
{J_{2,\mathrm{tail}}(1)}
\approx4.32\times10^{-3},
$
so only about $0.43\%$ of the initial tail-dissipation value remains at
the end of the displayed interval. The pronounced decrease seen in
Figure~\ref{fig:expC-objective} ({\bf Right}) therefore indicates that most of the remaining coordinate-$2$
nonlinear damping activity has already been exhausted by $T=8$.

\begin{figure}
	\begin{center}
		\includegraphics[width=8cm]{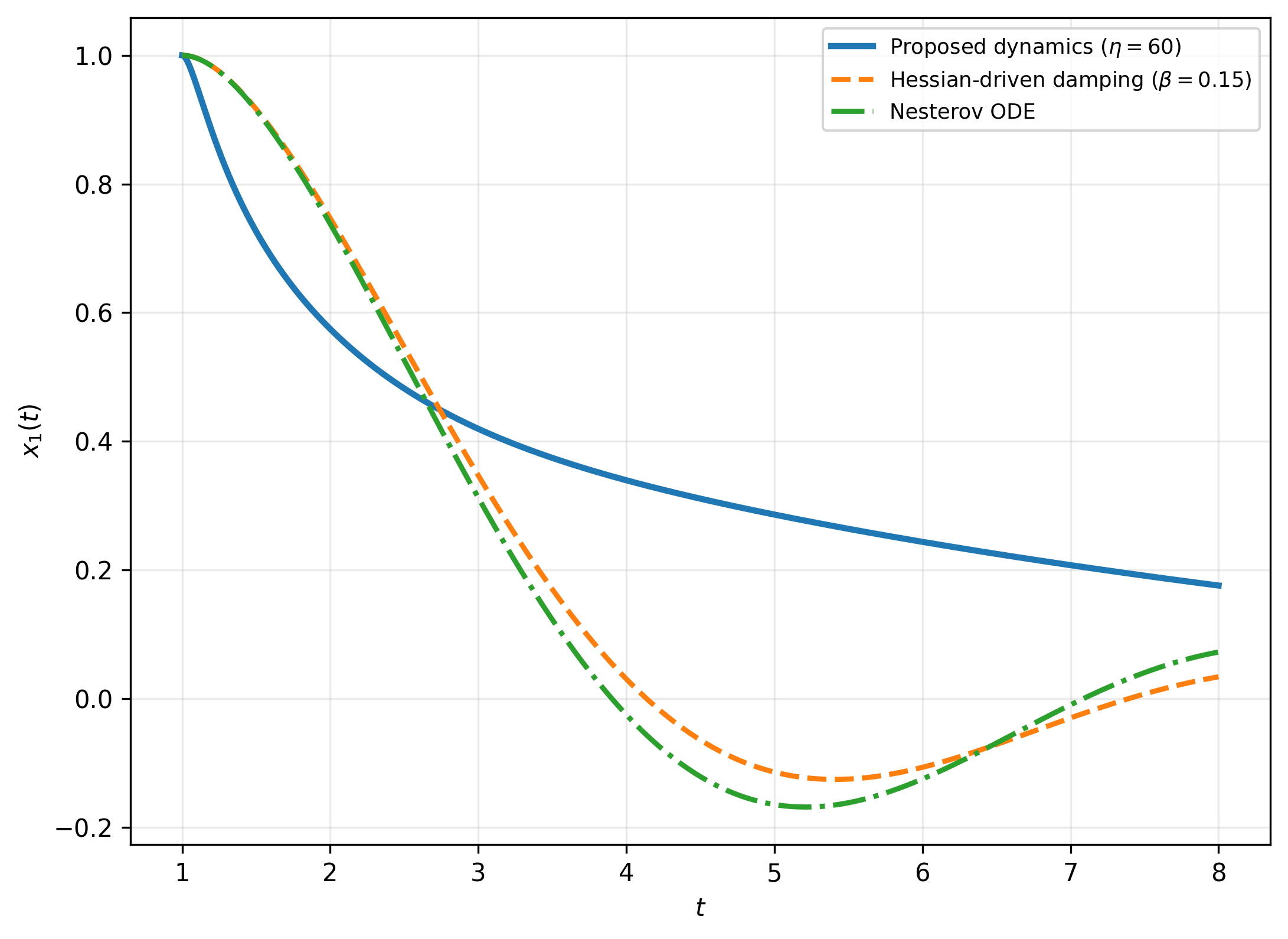}	
		\includegraphics[width=8cm]{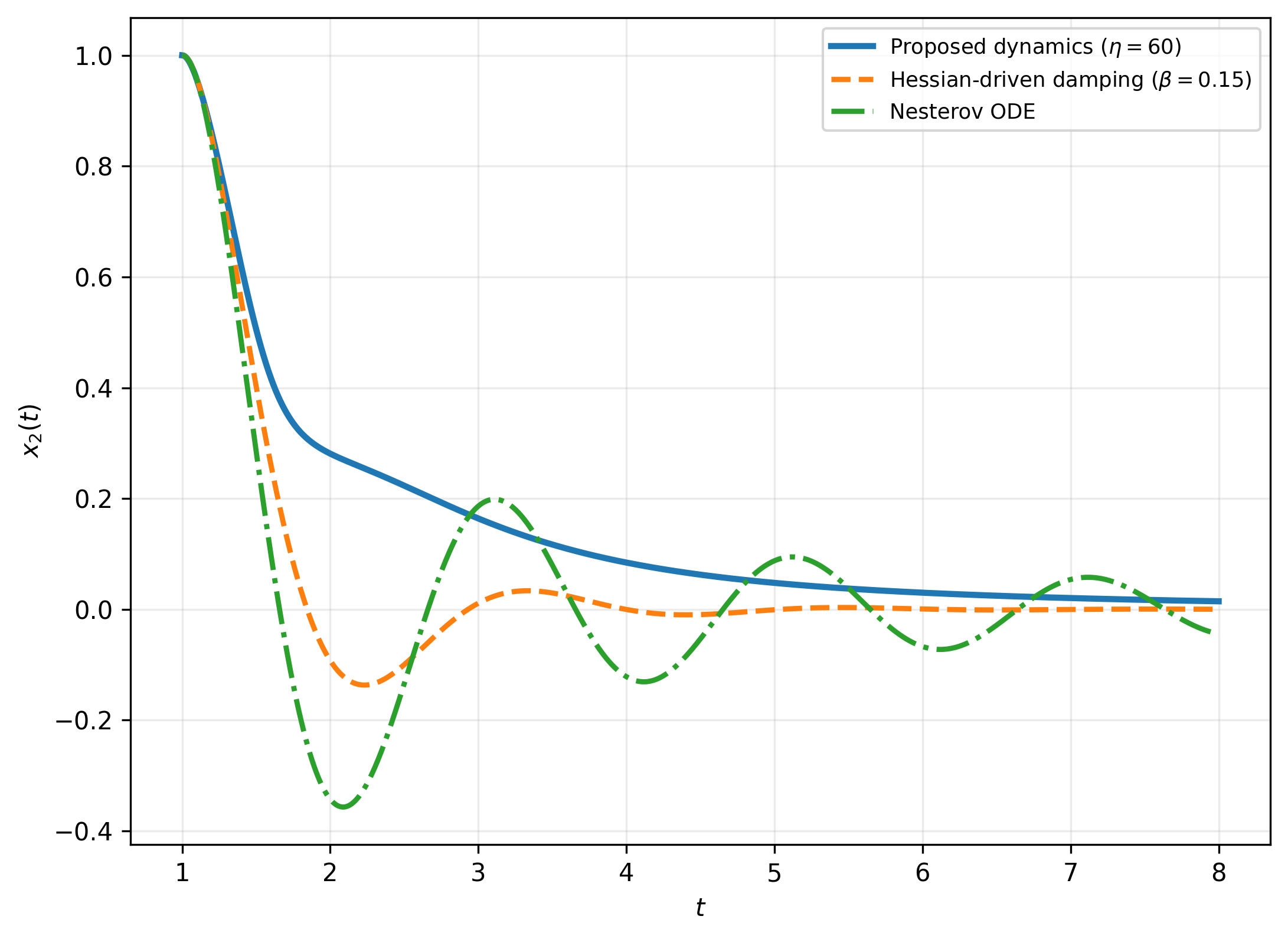}
	\end{center}
	\caption{{\bf Left:} First-coordinate trajectories. {\bf Right}: Second-coordinate trajectories. }\label{fig:expC-x1}
\end{figure}

\begin{figure}
	\begin{center}
		\includegraphics[width=8cm]{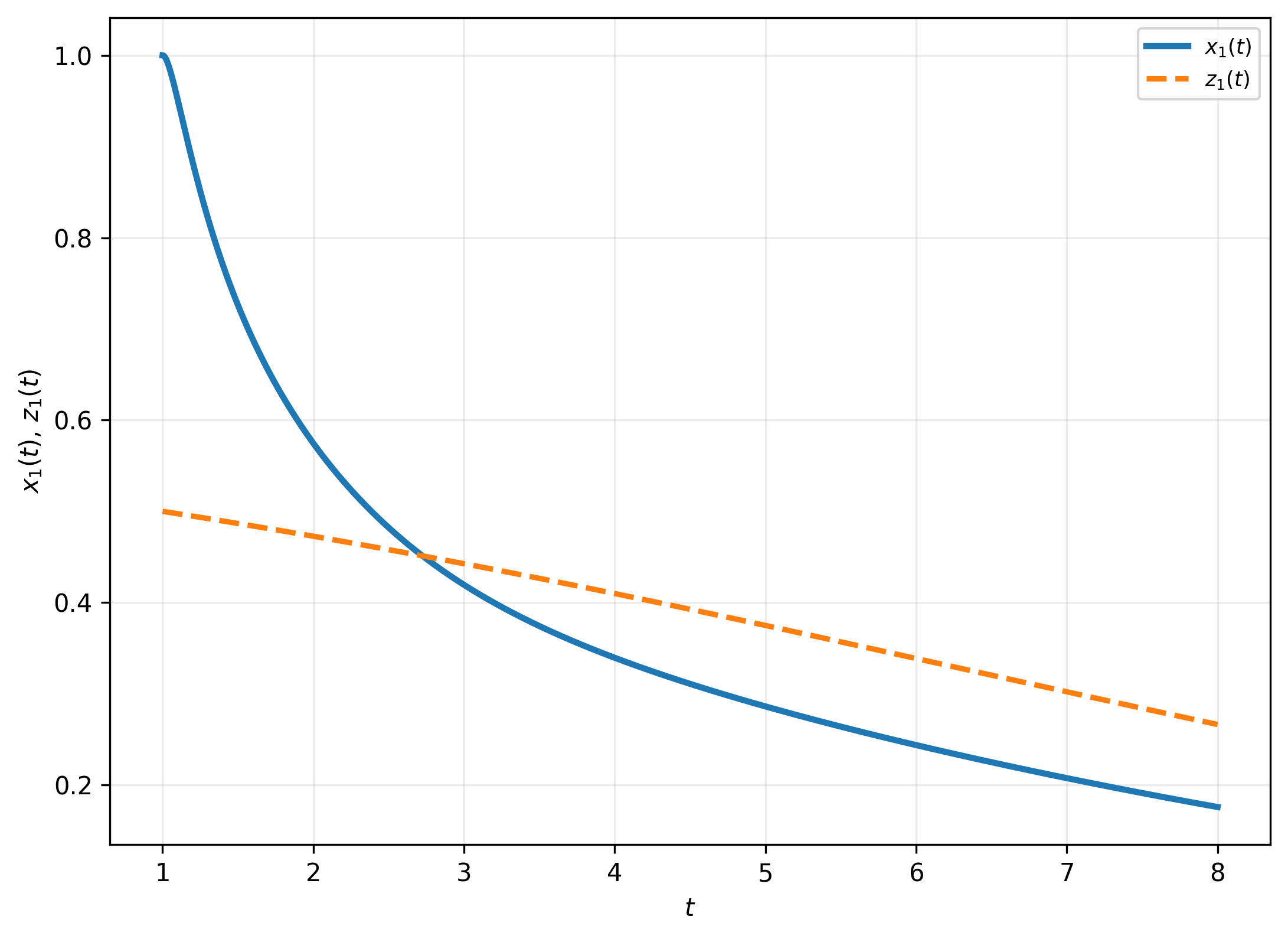}	
		\includegraphics[width=8cm]{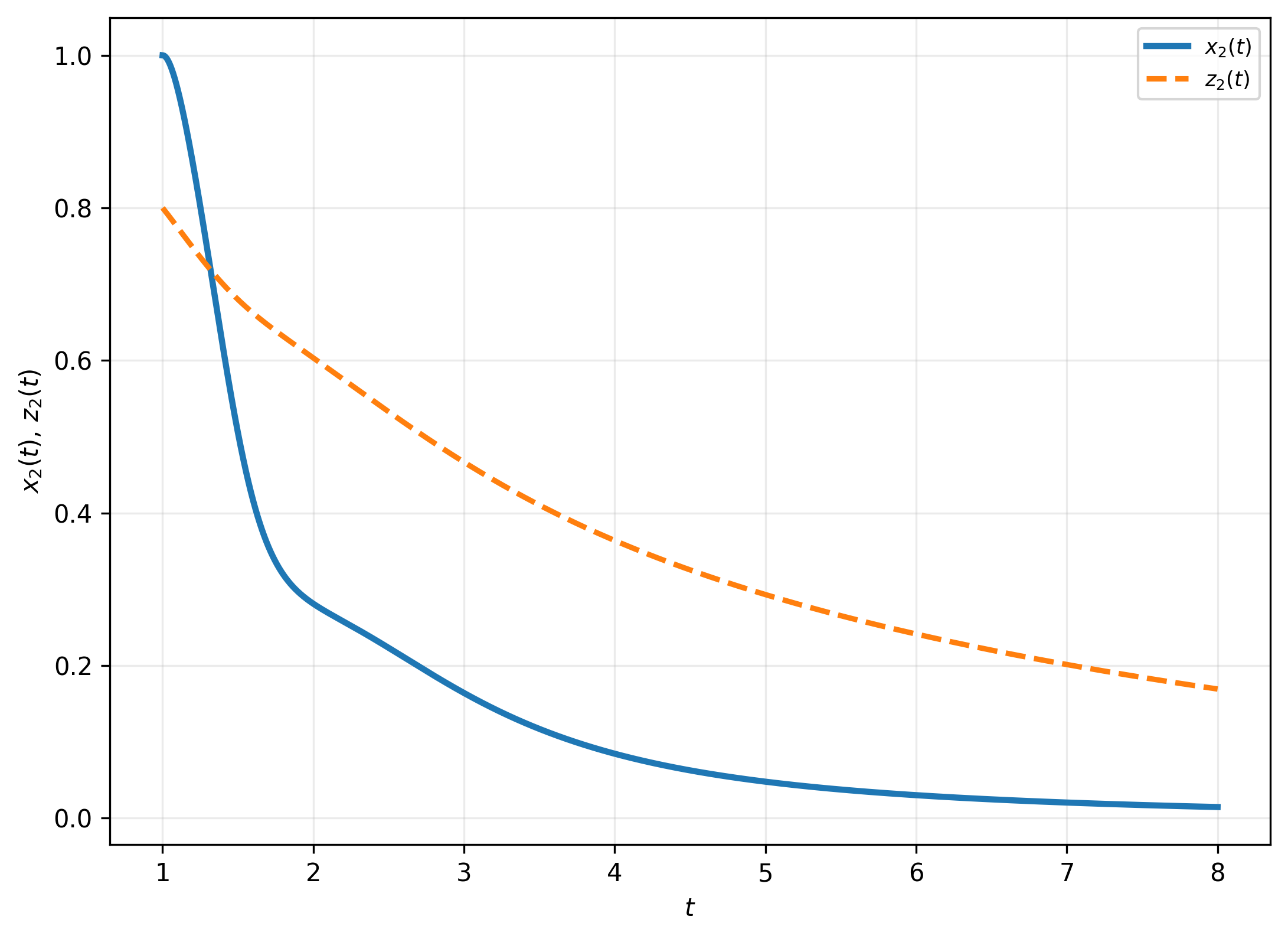}
	\end{center}
	\caption{{\bf Left:} Evolution of $x_1(t)$ and    $z_1(t)$. {\bf Right}: Evolution of $x_2(t)$ and $z_2(t)$.}\label{fig:expC-x1z1}
\end{figure}

\begin{figure}
	\begin{center}
		\includegraphics[width=8cm]{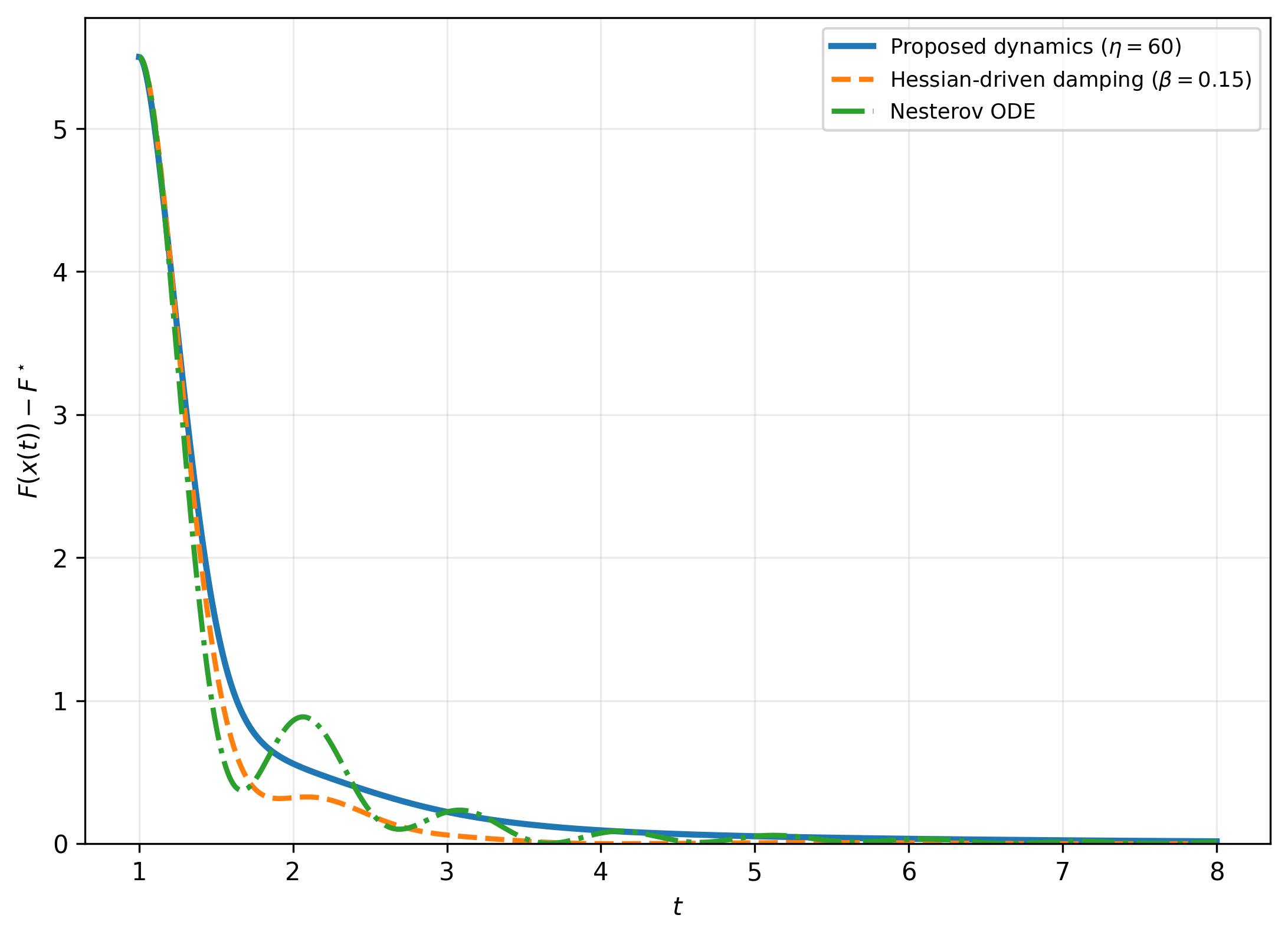}	
		\includegraphics[width=8cm]{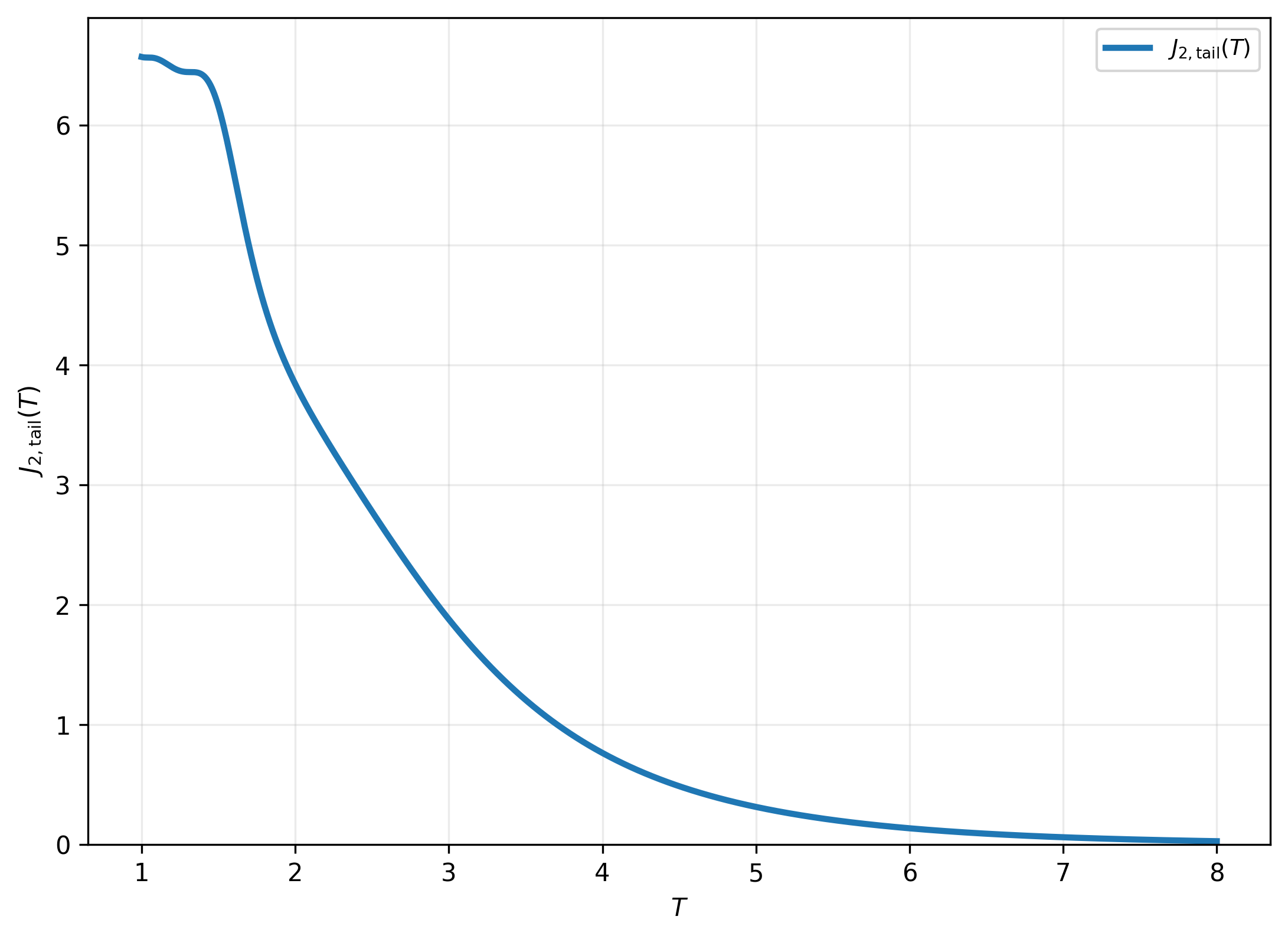}
	\end{center}
	\caption{{\bf Left:} Objective residual $F(x(t))-F^\star$.  {\bf Right}: Coordinate-$2$ tail dissipation. }\label{fig:expC-objective}
\end{figure}

\noindent In summary, although the Hessian-driven run attains a smaller residual at $t=8$, the proposed dynamics is the only run with smooth, sign-preserving trajectories, no zero crossings, and monotone residual on $[1,8]$. Thus, the advantage illustrated here is transient regularity and overshoot suppression rather than the smallest terminal residual.  

\section{Conclusion}\label{Con}

 This paper developed and analyzed an accelerated gradient flow with an endogenous gradient-memory anchor and a nonlinear restoring--damping feedback. The anchor is generated dynamically from a weighted history of the gradients, while the nonlinear feedback is modulated by the squared coordinatewise displacement from the anchor. Under the principal assumptions, we established global existence and uniqueness of strong solutions and constructed a Lyapunov functional that yields
$
F(x(t))-F^\star=\mathcal O(t^{-2}).
$
The nonlinear restoring--damping term enters the Lyapunov estimate with a favorable dissipative sign and therefore remains compatible with the accelerated objective-value rate. When $\gamma>0$, the primal trajectory converges strongly to a minimizer of $F$.

 A second feature of the dynamics is the role of the initial anchor in minimizer selection. When the solution set
$
S:=\argmin F
$
is a nonempty affine set, the limiting point is determined explicitly by
$
x(t)\longrightarrow P_S(z_0).
$
Thus, in the affine setting, the anchor does more than provide a moving reference state: its initial value determines which minimizer is selected. The rank-deficient least-squares case gives a concrete instance of this principle, with the dynamics converging to the least-squares solution closest to $z_0$.

 The nonlinear feedback also gives a quantitative displacement-sensitive dissipation mechanism. For $\eta>0$, we obtained a finite weighted coordinatewise phase-dissipation estimate. In addition, on intervals where an active coordinate remains separated from the anchor, the corresponding homogeneous phase response acquires an additional polynomial attenuation factor. These results give a precise description of the restoring--damping effect without requiring explicit Hessian information.

 The numerical experiments are consistent with these theoretical mechanisms. In the underdetermined least-squares example, changing the initial anchor while keeping the same primal initial data led to different limiting minimizers, in agreement with the projection-selection result. In the scalar example, varying the nonlinear parameter changed the transient profiles while the objective residuals remained consistent with the accelerated $\mathcal O(t^{-2})$ estimate. Finally, in the two-dimensional quadratic experiment, the proposed dynamics produced sign-preserving coordinate trajectories with no zero crossings on the displayed interval, whereas the Nesterov and Hessian-driven reference trajectories exhibited repeated zero crossings.  These observations illustrate how the anchor and restoring--damping mechanisms can influence both minimizer selection and transient behavior.

 Several questions remain open. It would be natural to determine whether the anchor-induced selection principle extends beyond affine solution sets. Another direction is to investigate whether stronger assumptions on $F$, such as error-bound conditions \cite{DrusvyatskiyLewis18}, Polyak--Lojasiewicz conditions \cite{Polyak63,KarimiNutiniSchmidt16,ApidopoulosGinattaVilla22}, or Kurdyka--{\L}ojasiewicz conditions \cite{Kurdyka98,BolteDaniilidisLewis07,AttouchBolteSvaiter13}, lead to sharper convergence rates for the distance to the solution set, the gradient norm, or the trajectory \cite{SebbouhDossalRondepierre20}. Finally, the development of discrete-time algorithms that preserve the selection and restoring--damping features of the continuous-time system is an important direction for future work.

\section*{Declarations}
\textbf{Conflict of interest}: There is no conflict of interest.

%\hfill

%\noindent \textbf{Acknowledgment}:

\noindent {\bf Authors' contributions:} All authors contributed equally and approved the final manuscript.

\hfill

\appendix

\section{Proof of Claim 1: \(\rho\in L^1(t_0,\infty)\)}\label{claim1}
The Lyapunov dissipation
inequality \eqref{eq:lyapunov-derivative-ineq-section4} gives
$$
\frac{d}{dt}E_{a,p}(t)
\le
-\gamma\|y(t)\|^2 .
$$
Integrating from \(t_0\) to \(T\), we obtain
$$
\gamma\int_{t_0}^{T}\|y(t)\|^2\,dt
\le
E_{a,p}(t_0)-E_{a,p}(T)
\le
E_{a,p}(t_0).
$$
Letting \(T\to\infty\), and using \(\gamma>0\), yields
$$
\int_{t_0}^{\infty}\|y(t)\|^2\,dt<\infty .
$$
Thus
$$
y\in L^2(t_0,\infty;\mathbb R^d).
$$
Since \(x(t)\) and \(z(t)\) are bounded, \(x(t)-z(t)\) is bounded.  Since
\(D_M(u)\) depends polynomially on \(u\), the matrix $
D_M(x(t)-z(t))$
is bounded on \([t_0,\infty)\).  Hence there exists a constant \(C_d>0\) such
that
$$
|\rho(t)|
\le
\frac{C_d}{t}\|y(t)\|.
$$
By Cauchy--Schwarz,
$$
\int_{t_0}^{\infty}|\rho(t)|\,dt
\le
C_d
\left(
\int_{t_0}^{\infty}\|y(t)\|^2\,dt
\right)^{1/2}
\left(
\int_{t_0}^{\infty}\frac{dt}{t^2}
\right)^{1/2}
<\infty.
$$

\hfill

\section{\Large \bf Proof of Existence and Uniqueness Result}\label{app:proofs-section4} 

For the reader's convenience, we collect here the proofs of the well-posedness
results stated in Section \ref{EU}.

\begin{proof}[Proof of Theorem \ref{thm:local-existence-section4}]
	Define
	$$
	y_0=\lambda(x_0-z_0)+t_0v_0.
	$$
	Set
	$$
	U(t)=\bigl(x(t),y(t),z(t)\bigr)\in\mathbb R^{3d}.
	$$
	For $t>0$, define the vector field
	$$
	\mathcal G(t,\xi,\omega,\zeta)
	=
	\begin{pmatrix}
		\displaystyle
		\frac{1}{t}\left[\omega-\lambda(\xi-\zeta)\right]
		\\[2mm]
		\displaystyle
		(1-2a)t\nabla F(\xi)
		-\gamma\omega
		-\frac{\eta}{t}D_M(\xi-\zeta)\omega
		\\[2mm]
		\displaystyle
		-\frac{t}{\lambda}(1-a)\nabla F(\xi)
	\end{pmatrix}.
	$$
	The mapping $u\mapsto D_M(u)$ is polynomial, and hence locally Lipschitz. Therefore
	$$
	(\xi,\omega,\zeta)
	\mapsto
	D_M(\xi-\zeta)\omega
	$$
	is locally Lipschitz on $\mathbb R^{3d}$. Since $\nabla F$ is locally Lipschitz on
	$\mathbb R^d$, it follows that
	$$
	(\xi,\omega,\zeta)\mapsto \mathcal G(t,\xi,\omega,\zeta)
	$$
	is locally Lipschitz on $\mathbb R^{3d}$, uniformly for $t$ in compact subintervals of
	$(0,\infty)$. Moreover, for each fixed $(\xi,\omega,\zeta)$, the map
	$$
	t\mapsto \mathcal G(t,\xi,\omega,\zeta)
	$$
	is continuous on $(0,\infty)$.\\
	
	\noindent By the Cauchy--Lipschitz theorem for nonautonomous ordinary differential equations
	\cite{CL,Hartman}, the initial-value problem
	$$
	\dot{U}(t)=\mathcal G(t,U(t)),
	\qquad
	U(t_0)=(x_0,y_0,z_0),
	$$
	admits a unique local solution
	$$
	U(t)=(x(t),y(t),z(t))\in C^1([t_0,T];\mathbb R^{3d})
	$$
	for some $T>t_0$.\\
	
	\noindent From \eqref{eq:first-order-section4}, we have
	$$
	\dot{x}(t)
	=
	\frac{1}{t}
	\left[
	y(t)-\lambda\bigl(x(t)-z(t)\bigr)
	\right].
	$$
	Since $y(t)$ and $z(t)$ are $C^1$, the right-hand side is $C^1$ on $[t_0,T]$. Hence
	$$
	x(t)\in C^2([t_0,T];\mathbb R^d).
	$$
	Using \eqref{eq:phase-section4}, \eqref{eq:first-order-section4}, and
	\eqref{eq:z-dynamics}, one recovers precisely \eqref{eq:system8}--\eqref{eq:z-dynamics}.
	Thus $(x(t),z(t))$ is a strong solution in the sense of Definition \ref{DF}.\\
	
	\noindent The same standard continuation theorem for finite-dimensional ordinary differential equations
	\cite{CL,Hartman} yields a maximal interval $[t_0,T_{\max})$. Suppose that
	$T_{\max}<\infty$ and that \eqref{eq:blowup-section4} fails. Then there exist
	a constant $C>0$ and a sequence $t_n\uparrow T_{\max}$ such that
	$$
	\|x(t_n)\|+\|\dot{x}(t_n)\|+\|z(t_n)\|
	\leq C.
	$$
	By \eqref{eq:phase-section4},
	$$
	y(t_n)=\lambda\bigl(x(t_n)-z(t_n)\bigr)+t_n\dot{x}(t_n),
	$$
	so $U(t_n)$ is bounded in $\mathbb R^{3d}$. Since $t_n\to T_{\max}$ and
	$T_{\max}>t_0>0$, the points $(t_n,U(t_n))$ remain in a compact subset of
	$(0,\infty)\times\mathbb R^{3d}$. On a fixed compact neighborhood of this set,
	$\mathcal G$ is bounded and locally Lipschitz in the state variable with a uniform
	Lipschitz constant. Hence the local existence time for the initial-value problems
	started at $(t_n,U(t_n))$ can be chosen uniformly positive for all sufficiently large
	$n$. For such $n$, this continuation extends the original solution beyond
	$T_{\max}$, contradicting maximality. Therefore \eqref{eq:blowup-section4} holds.
\end{proof}

\begin{proof}[Proof of Theorem \ref{thm:global-main-section4}]
	By Theorem \ref{thm:local-existence-section4}, there exists a unique maximal strong solution
	on $[t_0,T_{\max})$. Fix $x^\star\in\argmin F$. Thus, it suffices to prove that $T_{\max}=\infty$. In other words, we prove that the solution does not blow up in finite time.\\
	
	\noindent By Proposition \ref{prop:accelerated-lyapunov-a-section5}(i), applied on each $[t_0,T]\subset[t_0,T_{\max})$, we get
	\begin{equation*}
		E_{a,x^\star}(t)\le E_{a,x^\star}(t_0),
		\qquad
		t\in[t_0,T_{\max}).
	\end{equation*}
	In particular, there exist constants $C_y>0$ and $C_z>0$ such that
	\begin{equation}\label{eq:y-z-bounds-section4}
		\|y(t)\|\le C_y,
		\qquad
		\|z(t)-x^\star\|\le C_z,
		\qquad
		t\in[t_0,T_{\max}).
	\end{equation}
	
	\noindent It remains to prove that $x(t)$ is bounded. Define $
	h(t)=\|x(t)-x^\star\|^2$ and set $ R=
	\frac{C_y+\lambda C_z}{\lambda}.$
	Using \eqref{eq:first-order-section4}, we get
	$$
	\begin{aligned}
		\dot h(t)
		&=
		\frac{2}{t}
		\bigl\langle x(t)-x^\star,y(t)\bigr\rangle
		-
		\frac{2\lambda}{t}
		\|x(t)-x^\star\|^2
		+
		\frac{2\lambda}{t}
		\bigl\langle x(t)-x^\star,z(t)-x^\star\bigr\rangle.
	\end{aligned}
	$$
	Using \eqref{eq:y-z-bounds-section4}, we deduce that
	$$
	\dot h(t)
	\leq
	\frac{2\lambda}{t}\sqrt{h(t)}
	\left[
	R-\sqrt{h(t)}
	\right].
	$$
	Since $\lambda>0$, it follows that
	$$
	\dot h(t)<0
	\qquad
	\text{whenever } h(t)>R^2.
	$$
	
	\noindent We now prove that $h(t)$ is bounded on $[t_0,T_{\max})$. Set
	$$
	\rho:=\max\{h(t_0),R^2\}.
	$$
	We claim that
	$$
	h(t)\leq \rho
	\qquad
	\text{for all }t\in[t_0,T_{\max}).
	$$
	Suppose, by contradiction, that this claim is false. Then there exists
	$t_1\in[t_0,T_{\max})$ such that
	\(    h(t_1)>\rho.
	\)
	Since $h(t_0)\leq \rho<h(t_1)$ and $h$ is continuous, there exists a last time
	$\tau\in[t_0,t_1)$ such that
	$$
	h(\tau)=\rho
	$$
	and
	$$
	h(t)>\rho
	\qquad
	\text{for all }t\in(\tau,t_1].
	$$
	Since $\rho\geq R^2$, we have
	$$
	h(t)>R^2
	\qquad
	\text{for all }t\in(\tau,t_1].
	$$
	Therefore,
	$$
	\dot h(t)<0
	\qquad
	\text{for all }t\in(\tau,t_1].
	$$
	Integrating over $[\tau,t_1]$, we obtain
	$$
	h(t_1)-h(\tau)
	=
	\int_{\tau}^{t_1}\dot h(t)\,dt
	<0.
	$$
	Hence
	$$
	h(t_1)<h(\tau)=\rho,
	$$
	which contradicts the assumption that $h(t_1)>\rho$. Therefore,
	$$
	h(t)\leq \rho
	=
	\max\{h(t_0),R^2\}
	\qquad
	\text{for all }t\in[t_0,T_{\max}).
	$$
	Since $h(t)=\|x(t)-x^\star\|^2$, it follows that $x(t)$ is bounded on $[t_0,T_{\max})$.\\
	
	\noindent Since $x(t)$, $y(t)$, and $z(t)$ are bounded and $t\ge t_0>0$,
	\eqref{eq:first-order-section4} implies that $\dot{x}(t)$ is bounded on
	$[t_0,T_{\max})$. Therefore
	$$
	\|x(t)\|+\|\dot{x}(t)\|+\|z(t)\|
	$$
	cannot blow up in finite time. By the blow-up alternative in
	Theorem \ref{thm:local-existence-section4}, we must have
	$$
	T_{\max}=\infty.
	$$
	Therefore, the solution is global. This proves the theorem.
\end{proof}

\hfill

\begin{proof}[Proof of Proposition \ref{prop:a-one-section4}]
	Let $[t_0,T_{\max})$ be the maximal existence interval. Since $a=1$, equation \eqref{eq:z-dynamics} gives $    z(t)=z_0.$
	Recall  \eqref{eq:mechanical-energy-section5}:
	\begin{equation}\label{eq:a-one-energy-section4}
		\mathcal H(t)
		=
		\frac12\|\dot{x}(t)\|^2
		+
		F(x(t))-F_{\inf}
		+
		\frac{\gamma\lambda}{2t}\|r(t)\|^2
		+
		\frac{\eta\lambda}{4t^2}\Phi_M(r(t)).
	\end{equation}
	By Proposition \ref{prop:fixed-anchor-mechanical-section5}, applied on each $[t_0,T]\subset[t_0,T_{\max})$, we get
	\begin{equation}\label{eq:a-one-energy-bound-section4}
		\mathcal H(t)\le \mathcal H(t_0),
		\qquad
		t\in[t_0,T_{\max}).
	\end{equation}
	Since $F(x(t))-F_{\inf}\ge 0$, it follows from
	\eqref{eq:a-one-energy-section4} and \eqref{eq:a-one-energy-bound-section4} that
	$$
	\|\dot{x}(t)\|\le C
	$$
	for some constant $C>0$. Suppose, by contradiction, that $T_{\max}<\infty$.
	Then, for every $t\in[t_0,T_{\max})$,
	$$
	\|x(t)\|
	\le
	\|x(t_0)\|+C(T_{\max}-t_0).
	$$
	Thus $x(t)$ and $\dot{x}(t)$ are bounded on $[t_0,T_{\max})$. Since $z(t)=z_0$,
	the quantity
	$$
	\|x(t)\|+\|\dot{x}(t)\|+\|z(t)\|
	$$
	cannot blow up as $t\uparrow T_{\max}$, contradicting Theorem \ref{thm:local-existence-section4}.
	Hence $T_{\max}=\infty$. This proves the proposition.
\end{proof}

\hfill

\begin{proof}[Proof of Proposition \ref{prop:a-half-section4}]
	By Theorem \ref{thm:local-existence-section4}, there exists a unique maximal strong solution
	on $[t_0,T_{\max})$. When $a=1/2$, the phase equation
	becomes
	\begin{equation}\label{eq:a-half-phase-section4}
		\dot{y}(t)
		+
		\gamma y(t)
		+
		\frac{\eta}{t}D_M\bigl(x(t)-z(t)\bigr)y(t)
		=
		0.
	\end{equation}
	Taking the inner product of \eqref{eq:a-half-phase-section4} with $y(t)$ gives
	$$
	\frac{d}{dt}\frac12\|y(t)\|^2
	=
	-\gamma\|y(t)\|^2
	-
	\frac{\eta}{t}
	\bigl\langle
	D_M\bigl(x(t)-z(t)\bigr)y(t),
	y(t)
	\bigr\rangle.
	$$
	Since $\gamma\ge 0$, $\eta\ge 0$, and
	$$
	D_M\bigl(x(t)-z(t)\bigr)\succeq 0,
	$$
	we obtain
	$$
	\frac{d}{dt}\frac12\|y(t)\|^2\le 0.
	$$
	Therefore
	\begin{equation}\label{eq:y-bound-half-section4}
		\|y(t)\|\le \|y(t_0)\|,
		\qquad
		t\in[t_0,T_{\max}).
	\end{equation}
	Suppose, by contradiction, that $T_{\max}<\infty$. Since
	$t\in[t_0,T_{\max})$, using \eqref{eq:first-order-section4},
	\eqref{eq:y-bound-half-section4}, and $t\ge t_0>0$, there exists a constant $C>0$ such that
	\begin{equation}\label{eq:xdot-bound-half-section4}
		\|\dot{x}(t)\|
		\le
		C
		\left(
		1+\|x(t)\|+\|z(t)\|
		\right).
	\end{equation}
	Also, from \eqref{eq:z-dynamics} with $a=1/2$ and the linear-growth assumption
	\eqref{eq:linear-growth-section4},
	\begin{equation}\label{eq:zdot-bound-half-section4}
		\|\dot{z}(t)\|
		\le
		C
		\left(
		1+\|x(t)\|
		\right),
		\qquad t\in[t_0,T_{\max}).
	\end{equation}
	Combining these estimates gives
	\begin{equation}\label{eq:combined-bound-half-section4}
		\|\dot{x}(t)\|+\|\dot{z}(t)\|
		\le
		C
		\left(
		1+\|x(t)\|+\|z(t)\|
		\right).
	\end{equation}
	Let
	$$
	W(t)=\|x(t)\|+\|z(t)\|.
	$$
	Then, for every $t<T_{\max}$,
	$$
	W(t)
	\le
	W(t_0)
	+
	C\int_{t_0}^t
	\left(
	1+W(s)
	\right)\,ds.
	$$
	Gronwall's inequality yields a uniform bound for $W(t)$ on $[t_0,T_{\max})$, and
	\eqref{eq:xdot-bound-half-section4} gives the same for $\dot{x}(t)$. This contradicts
	the blow-up alternative in Theorem \ref{thm:local-existence-section4}. Hence
	$T_{\max}=\infty$.
\end{proof}

\end{document}